\documentclass[11pt,reqno]{amsart}
\usepackage[numbers, square]{natbib}
\usepackage{mathptmx}
\usepackage{amsmath}
\usepackage{amscd}
\usepackage{amssymb}
\usepackage{amsthm}
\usepackage{enumerate}
\usepackage{xspace}
\usepackage[all,tips]{xy}
\usepackage[dvips]{graphicx}
\usepackage{verbatim}
\usepackage{syntonly}
\usepackage{hyperref}
\usepackage{amsmath, amsthm, graphics, amssymb,fullpage,color, epsfig,url}
\usepackage{indentfirst}
\usepackage{color}

\theoremstyle{plain}
\newtheorem{thm}{Theorem}[section]
\newtheorem{lem}[thm]{Lemma}

\newtheorem{defn}[thm]{Definition}
\newtheorem{cor}[thm]{Corollary}
\newtheorem{clm}[thm]{Claim}

\theoremstyle{definition}
\newtheorem{rem}{Remark}[section]

\newcommand{\disp}{\displaystyle}

\DeclareMathOperator{\diam}{diam}
\DeclareMathOperator{\tr}{tr}

\DeclareMathOperator{\di}{div}

\newcommand{\eps}{\varepsilon}
\newcommand{\vp}{\varphi}

\newcommand{\al}{\alpha}
\newcommand{\be}{\beta}
\newcommand{\ga}{\gamma}
\newcommand{\de}{\delta}

\newcommand{\te}{\theta}
\newcommand{\la}{\lambda}

\newcommand{\om}{\omega}
\newcommand{\Om}{\Omega}
\newcommand{\si}{\sigma}

\newcommand{\ol}{\overline}

\newcommand{\nid}{\noindent}

\newcommand{\iny}{\infty}
\newcommand{\del}{ \partial}
\newcommand{\su}{\subset}
\newcommand{\LP}{\Delta}
\newcommand{\gr}{\nabla}

\newcommand{\norm}[1]{\left\vert \left\vert #1\right\vert\right\vert}

\newcommand{\abs}[1]{\left\vert#1\right\vert}
\newcommand{\set}[1]{\left\{#1\right\}}
\newcommand{\brac}[1]{\left[#1\right]}
\newcommand{\pr}[1]{\left( #1 \right) }

\newcommand{\N}{\ensuremath{\mathbb{N}}}

\newcommand{\R}{\ensuremath{\mathbb{R}}}

\numberwithin{equation}{section}
\title{The Landis Conjecture for variable coefficient second-order elliptic PDE\small{s}}

\author[Davey]{Blair Davey}
\address{Department of Mathematics, CCNY CUNY, New York, NY 10031, USA}
\email{bdavey@ccny.cuny.edu}
\author[Kenig]{Carlos Kenig}
\address{Department of Mathematics, University of Chicago, Chicago, IL 60637, USA}
\email{cek@math.uchicago.edu}
\thanks{Kenig is supported in part by DMS-1265429.}

\author[Wang]{Jenn-Nan Wang}
\address{Institute of Applied Mathematical Sciences, NCTS, National Taiwan University,
Taipei 106, Taiwan}
\email{jnwang@math.ntu.edu.tw}
\thanks{Wang is supported in part by MOST 102-2115-M-002-009-MY3.}

\date{}
\begin{document}
\maketitle

\begin{abstract}
In this work, we study the Landis conjecture for second-order elliptic equations in the plane. 
Precisely, assume that $V\ge 0$ is a measurable real-valued function satisfying $\norm{V}_{L^\iny\pr{\R^2}} \le 1$. Let $u$ be a real solution to $\di \pr{A \gr u} - V u = 0$ in $\R^2$. 
Assume that $\abs{u\pr{z}} \le \exp\pr{c_0 \abs{z}}$ and $u\pr{0} = 1$. 
Then, for any $R$ sufficiently large,
\[
\inf_{\abs{z_0} = R} \norm{u}_{L^\iny\pr{B_1\pr{z_0}}} \ge \exp\pr{- C R \log R}.
\]
In addition to equations with electric potentials, we also derive similar estimates for equations with magnetic potentials. 
The proofs rely on transforming the equations to Beltrami systems and Hadamard's three-quasi-circle theorem. 
\end{abstract}
\section{Introduction}

In this work, we study the asymptotic uniqueness for general second-order elliptic equations in the whole space.  One typical example we have in mind is
\begin{equation}
L u - V u := \di \pr{A \gr u} - V u = 0\quad\text{in}\quad\R^n,
\label{epde}
\end{equation}
where $A$ is symmetric and uniformly elliptic with Lipschitz continuous coefficients and $V$ is essentially bounded. For \eqref{epde}, we are interested in the following Landis type conjecture: assume that $\|V\|_{L^{\infty}(\R^n)}\le 1$ and $\|u\|_{L^{\infty}(\R^n)}\le C_0$ satisfies $|u(x)|\le C\exp(-C|x|^{1+})$, then $u\equiv 0$. When $L=\Delta$, counterexamples to the Landis conjecture were constructed by Meshkov in \cite{m92} where the exponent $4/3$ was shown to be optimal for complex-valued potentials and solutions. A quantitative form of Meshkov's result was derived by Bourgain and Kenig \cite{bk05} in their resolution of Anderson localization for the Bernoulli model in higher dimensions. The proof of Bourgain and Kenig's result was based on Carleman type estimates. Using the Carleman method, other related results for the general second elliptic equation involving the first derivative terms were obtained in \cite{D14} and \cite{LW14}. 

The known results mentioned above indicate that the exponent $1$ in the Landis type conjecture is not true for general coefficients and solutions. Therefore, we want to study the same question when $A$ and $V$ of \eqref{epde} are real-valued and the solution $u$ is also real. In the case where $L=\Delta$, $n=2$, and $V\ge 0$, a quantitative Landis conjecture was proved in \cite{KSW15}. Precisely, let $u$ be a real solution of $\Delta u-V u=0$ in $\R^2$ satisfying $u(0)=1, |u(x)|\le \exp(C_0|x|)$, where $\|V\|_{L^\infty}\le 1$ and $V\ge 0$.
Then for $R$ sufficiently large,
\[
\inf_{|x_0|=R}\sup_{|x-x_0|<1}|u(x)|\ge \exp(-CR\log R),
\]
where $C$ depends on $C_0$. 

Here we would like to generalize this result to the second-order elliptic operator $L$. 
Let $A$ be symmetric and uniformly elliptic with Lipschitz continuous coefficients. 
That is, for some $\lambda\in(0,1]$,
\begin{align}
&A = \brac{\begin{array}{ll} a_{11} & a_{12} \\ a_{21} & a_{22} \end{array}}, \, a_{12} = a_{21}
\label{symm} \\
& \la \abs{\xi}^2 \le a_{ij}\pr{x} \xi_i \xi_j \le \la^{-1} \abs{\xi}^2, \;\; \text{ for all } \, x \in \R^2, \xi \in \R^2.
\label{ellip} 
\end{align}
Since $A$ is Lipschitz continuous, then there exists $\mu > 0$ such that
\begin{align}
\norm{\gr a_{ij}}_{\iny} \le \mu \quad \text{ for each } \;\; i, j = 1, 2.
\label{gradDec}
\end{align}
The ellipticity condition \eqref{ellip} implies that
\begin{align}
& a_{ii} \ge \la \quad \text{ for each } \;\; i = 1, 2
\label{bndBelow} \\
& a_{ij} \le C\la^{-1} \quad \text{ for each } \;\; i, j = 1, 2.
\label{bndAbove}
\end{align}
We define the leading operator
\begin{equation}
L = \di\pr{A \gr}.
\label{Ldef}
\end{equation}

\begin{rem}
We will often use that $L$ is a divergence-form operator.
However, it will at times be useful to think of $L$ in non-divergence form:
$$L = \del_i\pr{a_{ij} \del_j u} = a_{ij} \del_{ij} u + \del_i a_{ij} \del_j u := a_{ij} \del_{ij} u + b_j \del_j u.$$
It follows from \eqref{gradDec} that $b \in L^\iny$ with $\norm{b_j}_{\iny} \le 2 \mu$ for each $j = 1, 2$.
\end{rem}

By building on the techniques developed in \cite{KSW15}, we will prove quantitative versions of Landis' conjecture when the leading operator is $L$.
As in \cite{KSW15}, to prove each Landis theorem, we first establish an appropriate order-of-vanishing estimate, then we apply the shift and scale argument from \cite{bk05}. 
We use the notation $B_r$ to denote a ball of radius $r$ centered at the origin.
As defined in Section \ref{quasi}, $Q_s$ denotes a quasi-ball of radius $s$ centered at the origin that is associated to an elliptic operator.
Constants $b$ and $d$ are chosen so that $B_b \su Q_1$ and $Q_{7/5} \Subset B_d$.
It is shown in Section \ref{quasi} that such ball exists, and they are bounded in terms of the ellipticity constant.
The functions $\si$ and $\rho$, which are introduced at the end of Section \ref{quasi} (see \eqref{2.1} and \eqref{2.2}), are used below to define $b$ and $d$.
The first maximal order-of-vanishing theorem that we will discuss is the following.

\begin{thm}
Set $b = \si\pr{1; \la}$, $d = \rho\pr{\frac{7}{5}; \la} + \frac{2}{5}$.
Let $u$ be a real-valued solution to
\begin{equation}
L u - V u = 0 \;\; \text{ in } B_d \subset \R^2,
\label{localPDE}
\end{equation}
where $V\ge 0$ and $A$ satisfies assumptions \eqref{symm} and \eqref{ellip}.
Assume that
\begin{align}
&\norm{u}_{L^\iny\pr{B_d }} \le \exp\pr{C_0 \sqrt{M}} 
\label{uBd} \\
& \norm{u}_{L^\iny\pr{B_b}} \ge 1 
\label{lowerBd} \\
& \norm{V}_{L^\iny\pr{B_d }} \le M
\label{VBd} \\
& \norm{\gr a_{ij}}_{L^\iny\pr{B_d }} \le \mu \sqrt{M},
\label{aGrBd} 
\end{align}
where $M \ge 1$.
Then there exists $C = C\pr{C_0, \la, \mu}$ so that
\begin{equation}
\norm{u}_{L^\iny\pr{B_r}} \ge r^{C \sqrt{M}}.
\label{smallBallEst}
\end{equation}
\label{localEst}
\end{thm}

As in \cite{bk05}, a scaling argument shows that the following quantitative form of Landis' conjecture follows from Theorem \ref{localEst}.

\begin{thm}
Assume that $V: \R^2 \to \R$ is measurable and satisfies
$$\norm{V}_{L^\iny\pr{\R^2}} \le 1.$$
Assume also that $V \ge 0$ a.e. in $\R^2$.
Let $u$ be a real solution to
\begin{equation}
L u - V u = 0 \,\,\text{ in }\,\,  \R^2,
\label{EPDE}
\end{equation}
where $A$ satisfies the assumptions \eqref{symm} -- \eqref{gradDec}.
Assume that $\abs{u\pr{z}} \le \exp\pr{c_0 \abs{z}}$ and $u\pr{0} = 1$, where $z = \pr{x,y}$.
Let $z_0 = \pr{x_0, y_0}$.
Then, for any $R$ sufficiently large,
\begin{align}
\inf_{\abs{z_0} = R} \norm{u}_{L^\iny\pr{B_1\pr{z_0}}} \ge \exp\pr{- C R \log R},
\label{LandisBd}
\end{align}
where $C$ depends on $c_0$, $\la$, $\mu$.
\label{Landis}
\end{thm}

The second maximal order-of-vanishing theorem applies to equations with a magnetic potential in divergence form.

\begin{thm}
Set $b = \si\pr{1; \la}$, $d = \rho\pr{\frac{7}{5}; \la} + \frac{2}{5}$.
Let $u$ be a real-valued solution to
\begin{equation}
L u + \gr \cdot \pr{W u} - V u = 0 \;\; \text{ in } B_d \subset \R^2,
\label{localPDEgrW}
\end{equation}
where $V\ge 0$ and $A$ satisfies assumptions \eqref{symm} and \eqref{ellip}.
Assume that for some $M \ge 1$, \eqref{uBd} -- \eqref{aGrBd} from above hold, and
\begin{align}
& \norm{W}_{L^\iny\pr{B_d}} \le \sqrt{M}.
\label{WBd}
\end{align}
Then there exists $C = C\pr{C_0, \la, \mu}$ such that \eqref{smallBallEst} holds.
\label{localEstgrW}
\end{thm}

As above, the order-of-vanishing estimate implies the following Landis result.

\begin{thm}
Assume that $V: \R^2 \to \R$, $W : \R^2 \to \R^2$ are measurable and satisfy
$$\norm{W}_{L^\iny\pr{\R^2}} \le 1, \;\; \norm{V}_{L^\iny\pr{\R^2}} \le 1.$$
Assume also that $V \ge 0$ a.e. in $\R^2$.
Let $u$ be a real solution to 
\begin{equation}
L u + \gr \cdot \pr{W  u} - V u = 0 \,\,\text{ in }\,\,  \R^2,
\label{EPDEgrW}
\end{equation}
where $A$ satisfies the assumptions \eqref{symm} -- \eqref{gradDec}.
Assume that $\abs{u\pr{z}} \le \exp\pr{c_0 \abs{z}}$ and $u\pr{0} = 1$, where $z = \pr{x,y}$.
Set $z_0 = \pr{x_0, y_0}$.
Then, for any $R$ sufficiently large, estimate \eqref{LandisBd} holds where $C$ depends on $c_0$, $\la$, $\mu$.
\label{LandisgrW}
\end{thm}

The third pair of theorems apply to equations with magnetic potentials in a non-divergence form.
{For this case, in the local setting, it suffices to work with matrices that have determinant equal to $1$}.  
This additional assumption changes the ellipticity constant, which in turn changes how we define $b$ and $d$.

\begin{thm}
Set $b = \si\pr{1; \la^2}$, $d = \rho\pr{\frac{7}{5}; \la^2} + \frac{2}{5}$.
Let $u$ be a real-valued solution to
\begin{equation}
L u - W \cdot \gr u - V u = 0 \;\; \text{ in } B_d \subset \R^2,
\label{localPDEngW}
\end{equation}
where $V\ge 0$ and $A$ satisfies assumptions \eqref{symm} and \eqref{ellip} with $\la$ replaced by $\la^2$, and $\det A = 1$.
Assume that for some $M \ge 1$, \eqref{uBd} -- \eqref{lowerBd}, and \eqref{aGrBd} from above hold, and
\begin{align}
& \norm{V}_{L^\iny\pr{B_d}} \le C_1 M 
\label{VCBd} \\
& \norm{W}_{L^\iny\pr{B_d}} \le \sqrt{C_1 M}.
\label{WCBd}
\end{align}
Then there exists $C = C\pr{C_0, C_1, \la, \mu}$ such that \eqref{smallBallEst} holds.
\label{localEstngW}
\end{thm}
\begin{rem}
{For the general coefficient matrix $A$ satisfying \eqref{symm} -- \eqref{gradDec}, dividing \eqref{localPDEngW} gives
\[
\di\pr{\frac{A}{\sqrt{\mbox{det}A}}\gr u}-\tilde{W}\cdot\gr u-\tilde{V}u=0,
\]
where
\begin{equation}\label{WWV}
\tilde W=A\gr\left(\frac{1}{\sqrt{\mbox{det}A}}\right)+\frac{W}{\sqrt{\mbox{det}A}},\quad\tilde V=\frac{V}{\sqrt{\mbox{det}A}}.
\end{equation}
If $W$ and $V$ satisfy \eqref{VCBd} and \eqref{WCBd}, then $\tilde W$ and $\tilde V$ satisfy the similar bounds with a new constant $C_1$ depending on $\lambda,\mu$. Also, the ellipticity constant of $A/\sqrt{\mbox{det}A}$ is $\lambda^2$.
}

\end{rem}
Again, the local theorem implies the Landis theorem. 

\begin{thm}
Assume that $V: \R^2 \to \R$, $W : \R^2 \to \R^2$ are measurable and satisfy
$$\norm{W}_{L^\iny\pr{\R^2}} \le 1, \;\; \norm{V}_{L^\iny\pr{\R^2}} \le 1.$$
Assume also that $V \ge 0$ a.e. in $\R^2$.
Let $u$ be a real solution to 
\begin{equation}
L u - W \cdot \gr u - V u = 0 \,\,\text{ in }\,\,  \R^2,
\label{EPDEngW}
\end{equation}
where $A$ satisfies the assumptions \eqref{symm} -- \eqref{gradDec}.
Assume that $\abs{u\pr{z}} \le \exp\pr{c_0 \abs{z}}$ and $u\pr{0} = 1$, where $z = \pr{x,y}$.
Set $z_0 = \pr{x_0, y_0}$.
{Then, for any $R\ge R_0$, estimate \eqref{LandisBd} holds, where $R_0$ depends on $\lambda,\mu$ and $C$ depends on $c_0$, $\la$, $\mu$}.
\label{LandisngW}
\end{thm}

This article is organized as follows.
In Section \ref{quasi}, we discuss fundamental solutions of second-order elliptic operators that satisfy \eqref{ellip}.
These results apply to second-order elliptic operators with $L^\iny$ coefficients.
These fundamental solutions lead to the definitions of quasi-balls and quasi-circles, as well as related results.
In Section \ref{SandS}, the shift and scale argument from \cite{bk05} is applied to show how each quantitative Landis theorem follows from the corresponding order-of-vanishing estimate.
A number of useful tools are developed in Section \ref{tools}.
To start, we introduce some first-order Beltrami operators that generalize $\ol \del$.
Then, a few properties that relate first-order Beltrami operators to second-order elliptic operators are established.
With these facts, a Hadamard three-quasi-circle theorem is proved.
Finally, we present some of the work of Bojarksi from \cite{Boj09} including a similarity principle for solutions to non-homogenous Beltrami equations.
In Section \ref{localEstProof}, the tools developed in the previous section are combined with the framework from \cite{KSW15} to prove Theorem \ref{localEst}.
Section \ref{localEstgrWProof} shows how to account for a magnetic potential, proving Theorem \ref{localEstgrW}.
The proof of Theorem \ref{localEstngW} is contained in Section \ref{localEstngWProof}.
A technical proof of one of the facts from Section \ref{tools} may be found in the appendix.

\section{Quasi-balls and quasi-circles}
\label{quasi}

Let $\mathcal{L}\pr{\la}$ denote the set of all second-order elliptic operators acting on $\R^2$ that satisfy ellipticity condition \eqref{ellip}.
Throughout this section, assume that $L \in \mathcal{L}\pr{\la}$.
We start by discussing the fundamental solutions of $L$.
These results are based on the Appendix of \cite{KN85}.

\begin{defn}
A function $G$ is called a fundamental solution for $L$ with pole at the origin if
\begin{itemize}
\item $G \in H^{1,2}_{loc}\pr{\R^2 \setminus 0}$, $G \in H^{1,p}_{loc}\pr{\R^2}$ for all $p < 2$ and for every $\vp \in C^\iny_0\pr{\R^2}$
$$\int a_{ij}\pr{z} D_i G\pr{z} \,  D_j \vp\pr{z} dz = - \vp\pr{0}.$$
\item $\abs{G\pr{z} } \le C \log \abs{z}$, for some $C > 0$, {$|z|\ge C$}.
\end{itemize}
\end{defn}

\begin{lem}[Theorem A-2, \cite{KN85}]
There exists a unique fundamental solution $G$ for $L$, with pole at the origin and with the property that $\disp \lim_{\abs{z} \to \iny} G\pr{z} - g\pr{z} = 0$, where $g$ is a solution to $L g = 0$ in $\abs{z} > 1$ with $g = 0$ on $\abs{z} = 1$.
Moreover, there are constants $C_1, C_2, C_3, C_4, R_1 < 1, R_2 > 1$, that depend on $\la$, such that
\begin{align*}
&C_1 \log\pr{\frac{1}{\abs{z}}} \le - G\pr{z} \le C_2 \log \pr{\frac{1}{\abs{z}}} \;\; \text{ for } \abs{z} < R_1 \\
& C_3 \log\abs{z} \le G\pr{z} \le C_4 \log \abs{z} \;\; \text{ for } \abs{z} > R_2.
\end{align*}
\label{fundSolBds}
\end{lem}

As a corollary to this theorem, we have the following.

\begin{cor}
There exist additional constants $C_5, C_6$, depending on $\la$, such that
\begin{align*}
& \abs{z}^{C_2} \le \exp\pr{G\pr{z}} \le \abs{z}^{C_1}  \;\; \text{ for } \abs{z} < R_1 \\
& C_5 \abs{z}^{C_2} \le \exp\pr{G\pr{z}} \le C_6 \abs{z}^{C_4} \;\; \text{ for } R_1 < \abs{z} < R_2 \\
& \abs{z}^{C_3} \le \exp\pr{G\pr{z}} \le \abs{z}^{C_4} \;\; \text{ for } \abs{z} > R_2.
\end{align*}
\label{expBdsFund}
\end{cor}

\begin{proof}
Exponentiating the bounds given in Theorem \ref{fundSolBds} gives the first and third line of inequalities.
Since $G$ is a solution to $L u = 0$ in the annulus $A = \set{z :  R_1 < \abs{z} < R_2}$, then by the maximum principle and the bounds given in {Lemma} \ref{fundSolBds}
\begin{align*}
& \max_{z \in A} G\pr{z} \le \max_{z \in \del A} G\pr{z} \le \max\set{C_4 \log R_2, C_1 \log R_1} = C_4 \log R_2 \\
& \min_{z \in A} G\pr{z} \ge  \min_{z \in \del A} G\pr{z} \ge \min\set{C_3 \log R_2, C_2 \log R_1} = C_2 \log R_1.
\end{align*}
It follows that for any $z \in A$,
\begin{align*}
& C_2 \log R_1 \le G\pr{z} \le C_4 \log R_2.
\end{align*}
Therefore, whenever $R_1 < \abs{z} < R_2$,
\begin{align*}
\exp\pr{G\pr{z}} \le R_2^{C_4} =  \pr{\frac{R_2}{\abs{z}}}^{C_4} \abs{z}^{C_4} \le   \pr{\frac{R_2}{R_1}}^{C_4} \abs{z}^{C_4},
\end{align*}
and
\begin{align*}
\exp\pr{G\pr{z}} \ge R_1^{C_2} = \pr{\frac{ R_1}{\abs{z}}}^{C_2} \abs{z}^{C_2} \ge \pr{\frac{ R_1}{R_2}}^{C_2} \abs{z}^{C_2},
\end{align*}
giving the second line of bounds.
\end{proof}

The level sets of $G$ will be important to us.

\begin{defn}
Define a function $\ell: \R^2 \to \pr{0, \iny}$ as follows: $\ell\pr{z} = s$ iff $G\pr{z} = \ln s$.
Then set 
\begin{align*}
Z_s &= \set{ z \in \R^2 : G\pr{z} = \ln s} 
= \set{ z \in \R^2 : \ell\pr{z} = s} .
\end{align*}
We refer to these level set of $G$ as {\bf quasi-circles.}
That is, $Z_s$ is the quasi-circle of radius $s$.
We also define (closed) {\bf quasi-balls} as
\begin{align*}
Q_s &= \set{ z \in \R^2 : \ell\pr{z} \le s} .
\end{align*}
Open {\bf quasi-balls} are defined analogously.
We may also use the notation $Q_s^L$ and $Z_s^L$ to remind ourselves of the underlying operator.
\end{defn}

The following lemma follows from the bounds given in Corollary \ref{expBdsFund}.

\begin{lem}
There are constants $c_1, c_2, c_3, c_4, c_5, c_6, S_1 < 1, S_2 > 1$, that depend on $\la$, such that if $z \in Z_s$, then
\begin{align*}
& s^{c_1} \le \abs{z} \le s^{c_2}  \;\; \text{ for } s \le S_1 \\
& c_5 s^{c_1} \le \abs{z} \le c_6 s^{c_4} \;\; \text{ for } S_1 < s < S_2 \\
& s^{c_3} \le \abs{z} \le s^{c_4} \;\; \text{ for } s \ge S_2.
\end{align*}
\label{ZsBounds}
\end{lem}

Thus, the quasi-circle $Z_s$ is contained in an annulus whose inner and outer radii depend on $s$ and $\la$.
For future reference, it will be helpful to have a notation for the bounds on these inner and outer radii.

\begin{defn}
Define
\begin{align}
& \si\pr{s; \la} = \sup\set{ r > 0 : B_r \su \bigcap_{L \in \mathcal{L}\pr{\la}} Q_s^L }\label{2.1}\\
& \rho\pr{s; \la} = \inf \set{r > 0 : \bigcup_{L \in \mathcal{L}\pr{\la}} Q_s^L \su B_r }.\label{2.2}
\end{align}
\end{defn}

\begin{rem}
These functions are defined so that for any operator $L$ in $\mathcal{L}\pr{\la}$, $B_{\si\pr{s; \la}} \su Q^L_s \su B_{\rho\pr{s;\la}}$.
\end{rem}

The quasi-balls and quasi-circles just defined above are centered at the origin since $G$ is a fundamental solution with a pole at the origin.
We may sometimes use the notation $Z_s\pr{0}$ and $Q_s\pr{0}$ as a reminder that these sets are centered around the origin.
If we follow the same process for any point $z_0 \in \R^2$, we may discuss the fundamental solutions with pole at $z_0$, and we may similarly define the quasi-circles and quasi-balls associated to these functions.
We will denote the quasi-circle and quasi-ball of radius $s$ centred at $z_0$ by $Z_s\pr{z_0}$ and $Q_s\pr{z_0}$, respectively.
Although $Q_s\pr{z_0}$ is not necessarily a translation of $Q_s\pr{0}$ for $z_0 \ne 0$, both sets are contained in annuli that are translations.

Throughout, we will often work with quasi-balls in addition to standard balls.

\section{The shift and scale arguments}
\label{SandS}

The bulk of the paper is devoted to proving the order-of-vanishing estimates stated in Theorems \ref{localEst}, \ref{localEstgrW}, and \ref{localEstngW}.
Before we get to those details, we show how Theorems \ref{Landis}, \ref{LandisgrW}, and \ref{LandisngW} follow from the local estimates and the shift and scale arguments in \cite{bk05}.

\begin{proof}[Proof of Theorem \ref{Landis}]
Let $u$ be a real-valued solution to \eqref{EPDE}.
Let $z_0 \in \R^2$ be such that $\abs{z_0} = R$ for some $R \ge 1$.
For a constant $a$ yet to be determined, define
$$u_R\pr{z} = u\pr{z_0 + a R z}, \quad A_R\pr{z} = A\pr{z_0 + a R z}, \quad V_R\pr{z} = \pr{a R}^2 V\pr{z_0 + a R z},$$
and set
$$L_R = \di \pr{A_R \gr}.$$
Since $A$ satisfies \eqref{symm} and \eqref{ellip}, then so too does $A_R$.
By construction, $u_R$ is a solution to
$$L_R u_R - V_R u_R = 0.$$
Since $\abs{u\pr{z}} \le \exp\pr{c_0 \abs{z}}$, it follows that
$$\norm{u_R}_{L^\iny\pr{B_d}} \le \exp\pr{c_0\pr{1 + a d} R},$$
where $d = \rho\pr{\frac{7}{5}; \la} + \frac{2}{5}$ depends on $\la$.
We choose $a > 0$ so that $\frac{1}{a} \le b$, where $b = \si\pr{1;\la}$ depends on $\la$.
Then $z_1 := - \frac{z_0}{a R} \in B_b$, $u_R\pr{z_1} = u\pr{0} = 1$ and it follows that
$$\norm{u_R}_{L^\iny\pr{B_b}} \ge 1.$$
Since $\norm{V}_{L^\iny} \le 1$, then $\norm{V_R}_{L^\iny\pr{B_d}} \le \pr{a R}^2$.
The condition $\norm{\gr a_{ij}}_{L^\iny} \le \mu$ implies that $\norm{\gr a_{R, ij}}_{L^\iny\pr{B_d}} \le a R \mu$.
Hence, the assumptions of Theorem \ref{localEst} are satisfied for $u_R$ with $M = \pr{a R}^2$.
Therefore,
\begin{equation*}
\norm{u_R}_{L^\iny\pr{B_r}} \ge r^{ C a R}.
\end{equation*}
Setting $r = \frac{1}{aR}$ and rewriting in terms of $u$, we see that
$$\norm{u}_{L^\iny\pr{B_1\pr{z_0}}} \ge \exp\pr{- \tilde C R \log R},$$
as required.
\end{proof}

\begin{proof}[Proof of Theorem \ref{LandisgrW}]
Let $u$ be a real-valued solution to \eqref{EPDEgrW}.
Define $z_0$, $a$, $u_R$, $A_R$, $V_R$, and $L_R$ as in the previous proof.
If we set
$$W_R\pr{z} = R \, W\pr{z_0 + a R z},$$
then $u_R$ is a solution to
$$L_R u_R + \gr\pr{W_R \, u_R} - V_R u_R = 0.$$
Since $\norm{W}_{L^\iny} \le 1$, then $\norm{W_R}_{L^\iny\pr{B_d}} \le aR$.
The assumptions of Theorem \ref{localEstgrW} are satisfied for $u_R$ with $M = \pr{a R}^2$, and the conclusion follows as above.
\end{proof}

To prove the third version of the theorem, we must account for the additional determinant condition in the statement of Theorem \ref{localEstngW}.

\begin{proof}[Proof of Theorem \ref{LandisngW}]
Let $u$ be a real-valued solution to \eqref{EPDEngW}.
Set $\tilde A = \frac{A}{\sqrt{\det A}}$ so that $\det \tilde A = 1$. 
Now the ellipticity constant of $\tilde A$ is $\lambda^2$.
Then $u$ is a solution to $\tilde L u - \tilde W \cdot \gr u - \tilde V u = 0$ in $\R^2$, where $\tilde L = \di \pr{\tilde A \gr}$ and $\tilde W$, $\tilde V$ are given in \eqref{WWV}. Note that $\norm{\tilde W}_{L^\iny} \le C_1$, and $\norm{\tilde V}_{L^\iny} \le \la^{-1}$, with $C_1 = C_1\pr{\la, \mu}$.
%Suppose $u$ is a solution to $$\di\pr{A \gr u} = W \cdot \gr u + V u.$$
%Set
%$$B = \brac{\begin{array}{cc} \frac{a_{11}}{\sqrt{\det A}}  & \frac{a_{12}}{ \sqrt{\det A}} \\ \frac{a_{12} }{\sqrt{\det A}} & \frac{a_{22}}{\sqrt{\det A}} \end{array}},$$
%and notice that $\det B =1$.
%Then
%\begin{align*}
%& \di \pr{B \gr u} \\
%&= \frac{1}{\sqrt{\det A}} \di \pr{A \gr u} 
%+ \brac{a_{11} \del_x\pr{\frac{1}{\sqrt{\det A}}} + a_{21} \del_y\pr{\frac{1}{\sqrt{\det A}}} } \del_x u 
%+ \brac{a_{12} \del_x\pr{\frac{1}{\sqrt{\det A}}} + a_{22} \del_y\pr{\frac{1}{\sqrt{\det A}}} } \del_y u \\
%&= \frac{W}{\sqrt{\det A}} \cdot \gr u + \frac{V}{\sqrt{\det A}} u + \tilde W \cdot \gr u,
%\end{align*}
%so $u$ satisfies an equation of the assumed form.
The rest of the proof proceeds as above.
%Let $z_0 \in \R^2$ be such that $\abs{z_0} = R$ for some $R \ge 1$.
%For a constant $a$ yet to be determined, define $u_R$ as above, then set
%$$\tilde A_R\pr{z} = \tilde A\pr{z_0 + a R z}, \quad \tilde V_R\pr{z} = \pr{a R}^2 \tilde V\pr{z_0 + a R z},\quad \tilde W_R\pr{z} = \pr{a R}^2 \tilde W\pr{z_0 + a R z}, \quad \tilde L_R = \di \pr{\tilde A_R \gr}.$$
%Since $A$ satisfies \eqref{symm} and \eqref{ellip}, then $\tilde A_R$ satisfies the same assumptions with $\la$ replaced by $\la^2$.
%By construction, $u_R$ is a solution to
%$$\tilde L_R u_R - \tilde W_R \cdot \gr u_R - \tilde V_R u_R = 0.$$
%As before, $\norm{u_R}_{L^\iny\pr{B_d}} \le \exp\pr{c_0\pr{1 + a d} R},$ where now $d = \rho\pr{\frac{7}{5}; \la^2} + \frac{2}{5}$.
%We choose $a > 0$ so that $\frac{1}{a} \le b = \si\pr{1;\la^2}$.
%Then $z_1 := - \frac{z_0}{a R} \in B_b$, $u_R\pr{z_1} = u\pr{0} = 1$ and it follows that $\norm{u_R}_{L^\iny\pr{B_b}} \ge 1.$
%Further, $\norm{\tilde V_R}_{L^\iny\pr{B_d}} \le \la^{-1}\pr{a R}^2$, $\norm{\tilde W_R}_{L^\iny\pr{B_d}} \le C_2 a R$, and $\norm{\gr a_{R, ij}}_{L^\iny\pr{B_d}} \le a R \mu$.
%Hence, the assumptions of Theorem \ref{localEstngW} are satisfied for $u_R$ with $M = \pr{C_1 a R}^2$, where $C_1$ depends on $\la$ and $\mu$.
%Therefore,
%\begin{equation*}
%\norm{u_R}_{L^\iny\pr{B_r}} \ge r^{ C C_1 a R}.
%\end{equation*}
%Setting $r = \frac{1}{aR}$ and rewriting in terms of $u$, we see that
%$$\norm{u}_{L^\iny\pr{B_1\pr{z_0}}} \ge \exp\pr{- \tilde C R \log R},$$
%as required.
\end{proof}
%
%
%The proof of Theorem \ref{LandisngW} in nearly identical to that of Theorem \ref{LandisgrW}, except that we use Theorem \ref{localEstngW} in place of Theorem \ref{localEstgrW}.

\section{Useful Tools}
\label{tools}

This section contains a number of tools that will be used in the proofs of the order-of-vanishing estimates to be given in the following sections.
We first define the Beltrami operator that will play the role of $\bar \del$ from \cite{KSW15}.
Then we present some results that show that such Beltrami operators {are} related to elliptic operators of the form $L$ in the same way that $\ol \del$ related to $\LP$.
These results are proved with elementary (but somewhat lengthly) computations.
Once we have the computational results, we will prove an optimal three-balls inequality, which we call the Hadamard three-quasi-ball inequality.
Finally, we present some work of Bojarski from \cite{Boj09}, including the similarity principle for equations of the form $D u = a u + b \bar u$.

\subsection{The Beltrami operators}

We define a Beltrami operator that will play the role of the $\bar \del$ operator from the original paper \cite{KSW15}.
For a complex-valued function $f = u + i v$, define 
\begin{align}
D f &=  \bar\del f + \eta\pr{z} \del f + \nu\pr{z} \ol{\del f} ,
\label{DDef} 
\end{align}
where 
\begin{align}
& \bar \del = \tfrac{1}{2} \pr{\del_x + i \del_y} \nonumber \\
&  \del = \tfrac{1}{2} \pr{\del_x - i \del_y} \nonumber \\
& \eta\pr{z} = \frac{a_{11} - a_{22} + 2 i a_{12} }{\det\pr{A + I}} 
\label{etaDef} \\
& \nu\pr{z} = \frac{\det A - 1}{\det\pr{A + I}} 
\label{nuDef}.
\end{align}

\begin{lem}
For $\eta, \nu$ defined above, we have
$$\abs{\eta\pr{z}} + \abs{\nu\pr{z}} \le \frac{1-\la}{1+\la}.$$
\label{etanuBds}
\end{lem}

\begin{proof}
The proof of this lemma is purely computation.
\begin{align*}
& \abs{\eta\pr{z}}^2 
= \frac{\pr{a_{11} - a_{22}}^2 + 4 a_{12}^2 }{\brac{\det\pr{A+I}}^2}
= \frac{\pr{a_{11} + a_{22}}^2 - 4a_{11}a_{22} + 4 a_{12}^2 }{\brac{\det\pr{A+I}}^2}
= \frac{\pr{\tr A}^2 - 4 \det A}{\pr{\det A + \tr A + 1}^2} \\
& \abs{\eta\pr{z}} 
= \frac{\la_1 - \la_2 }{\pr{\la_1 +1}\pr{ \la_2  + 1}} \\
%= \frac{\pr{\la_1 +1} - \pr{\la_2 +1}}{\pr{\la_1 + 1}\pr{\la_2 + 1}}
%=\frac{1}{\la_2 + 1} - \frac{1}{\la_1 + 1} \\
& \abs{\nu\pr{z}}
= \frac{\abs{\det A -1} }{\brac{\det\pr{A+I}}} 
= \frac{\abs{\la_1 \la_2 - 1}}{\pr{\la_1 +1}\pr{ \la_2  + 1}},
\end{align*}
where we are using $\la_1\ge\la_2$ to denote the eigenvalues of $A$.
It follows that
\begin{align*}
\abs{\eta\pr{z}} + \abs{\nu\pr{z}}
&= \frac{\la_1 - \la_2 }{\pr{\la_1 +1}\pr{ \la_2  + 1}} + \frac{\abs{\la_1 \la_2 - 1}}{\pr{\la_1 +1}\pr{ \la_2  + 1}} 
%&= \frac{\pr{\la_1 -1}\pr{\la_2 +1}}{\pr{\la_1 +1}\pr{\la_2 + 1}} \\
\le \frac{1-\la}{1+\la}.
\end{align*}
\end{proof}

A computation shows that for $f = u + iv$
\begin{align}
D f %= D \pr{u + iv} 
&= \frac{\pr{a_{11}  + \det A} + i a_{12}}{\det\pr{A + I}} u_x 
+ \frac{  a_{12} + i \pr{a_{22} + \det A}}{\det\pr{A + I}}  u_y 
\nonumber \\
&+ \frac{ \pr{a_{11} + 1} + i a_{12} }{\det\pr{A + I}} i v_x
+ \frac{a_{12} + i \pr{a_{22} + 1 } }{\det\pr{A + I}} i v_y .
\label{DExpr}
\end{align}
When $A$ has determinant equal to $1$, $\nu\pr{z} = 0$ and we may write
\begin{align}
D = \frac{\pr{a_{11}  + 1} + i a_{12}}{\det\pr{A + I}} \del_x + \frac{  a_{12} + i \pr{a_{22} + 1}}{\det\pr{A + I}}  \del_y.
\label{DDet1Expr}
\end{align}

In addition to the operator $D$, we will also make use of an operator that is related to $D$ through some function $w$.
For a given function $w$, set
$$\eta_w\pr{z} = \left\{\begin{array}{ll} 
\eta\pr{z} + \nu\pr{z} \frac{\ol{\del w}}{\del w} & \text{ for } \del w \ne 0  \\  
\eta\pr{z} + \nu\pr{z} & \text{ otherwise }
\end{array} \right.,$$
where $\eta$ and $\nu$ are as defined in \eqref{etaDef} and \eqref{nuDef}, respectively.
By Lemma \ref{etanuBds}, it follows that $\disp \abs{\eta_w} \le \frac{1 - \la}{1+\la}$.
Define
\begin{align}
D_w f = \ol{\del} f + \eta_w\pr{z} \del f.
\label{DwDef}
\end{align}
If $\eta_w\pr{z} = \al_w\pr{z} + i \be_w\pr{z}$, then
\begin{align}
D_w &= \frac{1}{2} \brac{\del_x + i \del_y + \pr{\al_w + i \be_w} \pr{\del_x - i \del_y } }
\nonumber \\
&= \frac{1 + \al_w + i \be_w}{2} \del_x + \frac{\be_w + i\pr{1 - \al_w}}{2} \del_y
\end{align}
Bertrami operators of this form will be used in the proofs of the main theorems.

At times, the dependence on $w$ will not be important to our arguments, so we define
\begin{equation}
\hat D = \frac{1 + \al + i \be}{2} \del_x + \frac{ \be + i \pr{1 - \al}}{2} \del_y,
\label{hatDDef}
\end{equation}
where $\al, \be$ are assumed to be functions of $z$ such that $\disp \al^2 + \be^2 \le \pr{\frac{1 - \la}{1+\la}}^2 < 1$.
Associated to $\hat D$ is the second-order elliptic operator $\hat L = \di \pr{ \hat A \gr}$ with
\begin{equation}
\hat A 
= \brac{\begin{array}{ll} \frac{\pr{1+\al}^2 + \be^2}{1 - \al^2 - \be^2} & \frac{2 \be}{1 - \al^2 - \be^2} \\ \frac{2\be}{1 - \al^2 - \be^2} &  \frac{\pr{1-\al}^2 + \be^2}{1 - \al^2 - \be^2} \end{array}}
= \brac{\begin{array}{ll} \hat a_{11} & \hat a_{12} \\ \hat a_{12} & \hat a_{22} \end{array}}.
\label{hatADef}
\end{equation}
A computation shows that the smallest eigenvalue of $\hat A$ satisfies
\begin{align*}
\la_- 
%&= \brac{\frac{1 + \al^2 + \be^2 - \sqrt{\pr{1 + \al^2 + \be^2}^2 - \pr{1 - \al^2 - \be^2}^2}}{1 - \al^2 - \be^2}} \\
%&= \brac{\frac{1 + \al^2 + \be^2 - 2\sqrt{\al^2 + \be^2}}{1 - \al^2 - \be^2}} \\
&= 1 - \frac{2}{1 + \pr{\al^2 + \be^2}^{-1/2}}=\frac{1-\sqrt{\alpha^2+\beta^2}}{1+\sqrt{\alpha^2+\beta^2}}
\ge \la,
\end{align*}
while the largest eigenvalue of $\hat A$ satisfies
\[
\lambda_+=\frac{1+\sqrt{\alpha^2+\beta^2}}{1-\sqrt{\alpha^2+\beta^2}}\le\lambda^{-1}.
\]
Therefore we can see that $\hat A$ has the same ellipticity constant, $\la$.
Finally, note that if $\det A = 1$, then $D$ takes the form of $\hat D$.
This means that the rest of the results of this section may be applied to $D$ in this case.
\begin{rem}\label{rem4.1}
{Note that if $D$ is given as in \eqref{DDef} and $Df=0$, then $D_wf=0$ with $w=f$, where $D_w$ is defined in \eqref{DwDef}.}
\end{rem}

\subsection{Computational results for elliptic operators}

The following results show that $\hat D$ relates to $\hat L$ in some of the same ways that $\ol \del$ relates to $\LP$.
These properties will allow us to prove the Hadamard three-quasi-circle theorem.

\begin{lem}
If $\hat D f = 0$, where $f\pr{x,y} = u\pr{x,y} + i v\pr{x,y}$ for real-valued $u$ and $v$, then
\begin{align*}
& \hat L u = 0 = \hat Lv.
\end{align*}
\label{rpLem}
\end{lem}

\begin{proof}
If $\hat D f = 0$, then it follows from \eqref{hatDDef} that the following Cauchy-Riemann type equations hold
\begin{equation}
\left\{ \begin{array}{l}  
\pr{1 + \al} u_x - \be v_x + \be u_y - \pr{1 - \al} v_y = 0 \\
\be u_x + \pr{1 + \al} v_x + \pr{1 - \al} u_y + \be v_y = 0.
\end{array}\right.
\label{CRType}
\end{equation}
Some algebraic manipulations give rise to two more equivalent sets of equations
\begin{equation}
\left\{ \begin{array}{l}  
\hat a_{11} u_x + \hat a_{12} u_y -  v_y = 0 \\
\hat a_{12} u_x +  \hat a_{22} u_y + v_x = 0,
\end{array}\right.
\label{CR4Eq1}
\end{equation}
and
\begin{equation}
\left\{ \begin{array}{l}  
\hat a_{11} v_x + \hat a_{12} v_y + u_y = 0 \\
\hat a_{12} v_x + \hat a_{22} v_y - u_x = 0,
\end{array}\right.
\label{CR4Eq2}
\end{equation}
where we have used the definition of $\hat A$ in \eqref{hatADef}.
From \eqref{CR4Eq1}, we have
\begin{align*}
0 &= \del_x\brac{\hat a_{11} u_x + \hat a_{12} u_y - v_y} + \del_y \brac{\hat a_{12} u_x + \hat a_{22} u_y + v_x},
\end{align*}
so that $\hat L u = 0$.
Similarly, by \eqref{CR4Eq2},
\begin{align*}
0 &= \del_x\brac{\hat a_{11} v_x + \hat a_{12} v_y + u_y} + \del_y \brac{\hat a_{12} v_x + \hat a_{22} v_y - u_x},
\end{align*}
so that $\hat L v = 0$ as well.
\end{proof}

We find another parallel with the Laplace equation.  
As in the case of $\hat L = \LP$, the logarithm of the norm of ${f}$ is a subsolution to the second-order equation whenever $\hat D f = 0$. 
To see this, it suffices to prove that

\begin{lem}
If $\hat D f = 0$ and $f \ne 0$, then $\hat L\brac{ \log\abs{f\pr{z}}} = 0$.
\label{logLem}
\end{lem}

\begin{proof}
If $f = u + i v$, where $u$ and $v$ are real-valued, then $\log\abs{f\pr{z}} = \frac{1}{2} \log\pr{u^2 + v^2}$.
We have
\begin{align*}
\del_x \log\abs{f\pr{z}} &= \frac{u u_x + v v_x}{u^2 + v^2} \\
\del_y \log\abs{f\pr{z}} &= \frac{u u_y + v v_y}{u^2 + v^2}. %\\
%\del_{xx} \log\abs{f\pr{z}} &= \frac{\pr{u_x}^2 + \pr{v_x}^2 + u u_{xx} + v v_{xx}}{u^2 + v^2} - 2\frac{\pr{u u_x + v v_x}^2}{\pr{u^2 + v^2}^2} \\
%\del_{yy} \log\abs{f\pr{z}} &= \frac{\pr{u_y}^2 + \pr{v_y}^2 + u u_{yy} + v v_{yy}}{u^2 + v^2} - 2\frac{\pr{u u_y + v v_y}^2}{\pr{u^2 + v^2}^2} \\
\end{align*}
Then,
\begin{align*}
\hat L [\log \abs{f\pr{x}} ]
&= \del_x \pr{\hat a_{11}\del_x \log\abs{f\pr{z}} + \hat a_{12}\del_y \log\abs{f\pr{z}}} 
+ \del_y \pr{\hat a_{12}\del_x \log\abs{f\pr{z}} + \hat a_{22}\del_y \log\abs{f\pr{z}}} \\
&=\del_x \pr{\hat a_{11}\frac{u u_x + v v_x}{u^2 + v^2} + \hat a_{12}\frac{u u_y + v v_y}{u^2 + v^2}} 
+ \del_y \pr{\hat a_{12}\frac{u u_x + v v_x}{u^2 + v^2} + \hat a_{22}\frac{u u_y + v v_y}{u^2 + v^2}} \\
&=\brac{\del_x \pr{\hat a_{11}u_x + \hat a_{12}u_y} + \del_y \pr{\hat a_{12}u_x+ \hat a_{22}u_y}}\frac{u }{u^2 + v^2} \\
&+\brac{\del_x \pr{\hat a_{11}v_x + \hat a_{12} v_y} + \del_y \pr{\hat a_{12} v_x + \hat a_{22} v_y}}\frac{v}{u^2 + v^2} \\
&+ \del_x \pr{\frac{u }{u^2 + v^2}} \hat a_{11} u_x +\del_x \pr{\frac{u }{u^2 + v^2}} \hat a_{12} u_y
+ \del_y \pr{\frac{u }{u^2 + v^2} }\hat a_{12} u_x + \del_y \pr{\frac{u }{u^2 + v^2}}\hat a_{22} u_y \\
&+\del_x \pr{\frac{v }{u^2 + v^2}}\hat a_{11} v_x + \del_x\pr{\frac{v }{u^2 + v^2}}\hat a_{12} v_y
+ \del_y \pr{\frac{v }{u^2 + v^2}}\hat a_{12}v_x + \del_y\pr{\frac{ v }{u^2 + v^2}}\hat a_{22}v_y. %\\
%&= \del_x \pr{\frac{u }{u^2 + v^2}} a_{11} u_x +\del_x \pr{\frac{u }{u^2 + v^2}} a_{12} u_y
%+ \del_y \pr{\frac{u }{u^2 + v^2} }a_{12} u_x + \del_y \pr{\frac{u }{u^2 + v^2}}a_{22} u_y \\
%&+\del_x \pr{\frac{v }{u^2 + v^2}}a_{11} v_x + \del_x\pr{\frac{v }{u^2 + v^2}}a_{12} v_y
%+ \del_y \pr{\frac{v }{u^2 + v^2}}a_{12}v_x + \del_y\pr{\frac{ v }{u^2 + v^2}}a_{22}v_y,
\end{align*}
Since $\hat L u = 0 = \hat Lv$ by the previous lemma, the top two lines vanish and we have,
\begin{align*}
\hat L [\log \abs{f\pr{x}}] 
&= \del_x \pr{\frac{u }{u^2 + v^2}} \hat a_{11} u_x +\del_x \pr{\frac{u }{u^2 + v^2}} \hat a_{12} u_y
+ \del_y \pr{\frac{u }{u^2 + v^2} }\hat a_{12} u_x + \del_y \pr{\frac{u }{u^2 + v^2}}\hat a_{22} u_y \\
&+\del_x \pr{\frac{v }{u^2 + v^2}}\hat a_{11} v_x + \del_x\pr{\frac{v }{u^2 + v^2}}\hat a_{12} v_y
+ \del_y \pr{\frac{v }{u^2 + v^2}}\hat a_{12}v_x + \del_y\pr{\frac{ v }{u^2 + v^2}}\hat a_{22}v_y \\
&= \set{\hat a_{11} \brac{\pr{u_x}^2+\pr{v_x}^2} + 2 \hat a_{12}\pr{u_x u_y + v_x v_y} + \hat a_{22}\brac{ \pr{u_y}^2 + \pr{v_y}^2 }}\frac{1}{u^2 + v^2} \\
&-2 {\brac{\hat a_{11}\pr{uu_x + v v_x}^2 + 2\hat a_{12}\pr{u u_x + v v_x}\pr{u u_y + v v_y} + \hat a_{22}\pr{u u_y + v v_y }^2} }\frac{1 }{\pr{u^2 + v^2}^2}.
\end{align*}
By the relations \eqref{CR4Eq1} and \eqref{CR4Eq2},
\begin{align*}
& \hat a_{11} \brac{\pr{u_x}^2+\pr{v_x}^2} + 2 \hat a_{12}\pr{u_x u_y + v_x v_y} + \hat a_{22}\brac{ \pr{u_y}^2 + \pr{v_y}^2} \\
&= \hat a_{11} \brac{\pr{u_x}^2+\pr{v_x}^2} 
+ \brac{u_x\pr{v_y - \hat a_{11} u_x} + v_x\pr{-u_y - \hat a_{11}v_x}} \\
&+\brac{ \pr{-v_x - \hat a_{22} u_y} u_y + \pr{u_x - \hat a_{22} v_y} v_y}
+  \hat a_{22}\brac{ \pr{u_y}^2 + \pr{v_y}^2} \\
&= 2 \pr{u_x v_y - v_x u_y},
\end{align*}
and
\begin{align*}
& \hat a_{11}\pr{uu_x + v v_x}^2 + 2\hat a_{12}\pr{u u_x + v v_x}\pr{u u_y + v v_y} + \hat a_{22}\pr{u u_y + v v_y }^2 \\
&= \pr{uu_x + v v_x}\set{ \hat a_{11}\pr{uu_x + v v_x}
+ \brac{u \pr{v_y- \hat a_{11} u_x} + v\pr{- u_y  - \hat a_{11} v_x}}} \\
&+ \set{\brac{u\pr{- v_x - \hat a_{22} u_y}+ v \pr{u_x  - \hat a_{22} v_y}} + \hat a_{22}\pr{u u_y + v v_y}} \pr{u u_y + v v_y} \\
&= \pr{u u_x + v v_x}\pr{u v_y -v u_y} + \pr{-u v_x + v u_x}\pr{u u_y + v v_y} \\
%&= \pr{u_x v_y - u_y v_x} u^2 + \pr{v_x v_y - u_x u_y - v_x v_y + u_x u_y} uv + \pr{- u_y v_x + u_x v_y} v^2 \\
&= \pr{u_x v_y - u_y v_x} \pr{u^2 + v^2}.
\end{align*}
%since $\det \, A = a_{11} a_{22} - a_{12}^2 = 1$.
Therefore,
\begin{align*}
\hat L [\log \abs{f\pr{z}}] &= 0,
\end{align*}
proving the lemma.
\end{proof}

Since $\LP = 4 \bar \del \del = 4 \del \bar \del$ is used in \cite{KSW15} to prove the third version of the theorem, we would like a decomposition for our operator $L = \di\pr{A \gr}$ into first-order operators.
Under some additional assumptions on the structure of $A$, the following lemma shows that this is possible.

\begin{lem}
Assume that $A$ has determinant equal to $1$.
Then the operator $L$ may be decomposed as
$$L = \pr{D + \widetilde W} \widetilde D,$$
where
\begin{align*}
\widetilde D &= \brac{1 + a_{11} - i a_{12}} \del_x + \brac{a_{12} - i\pr{1+a_{22}}} \del_y 
= \det\pr{A+I} \overline D  \\
\widetilde W 
%&= \frac{w_1 + i w_2}{\det\pr{A+I}} \\
&= \frac{\pr{\al \del_x a_{11} - \be \del_x a_{12} + \ga \del_y a_{11} + \de \del_y a_{12}} + i \pr{\ga \del_x a_{11} + \de \del_x a_{12} - \al \del_y a_{11} + \be \del_y a_{12}}}{ a_{11} \det\pr{A+I}^2}  %\\
%\widetilde W_1 &= - \al \del_x a_{11} + \be \del_x a_{12} + \ga \del_y a_{11} + \de \del_y a_{12}, \\
%\widetilde W_2 &= \ga \del_x a_{11} + \de \del_x a_{12} + \al \del_y a_{11} - \be \del_y a_{12},
\end{align*}
\begin{align*}
&\al = a_{11} + a_{22} + 2 a_{11}a_{22} \qquad
\be = 2 a_{12}\pr{1 + a_{11}} \\
&\ga = a_{12}\pr{a_{22} - a_{11}} \qquad\qquad
\de = \pr{1 + a_{11}}^2 - a_{12}^2,
\end{align*}
\label{decompLem}
and $D$ is given by \eqref{DDet1Expr}.
\end{lem}

The proof of this lemma is straightforward, but tedious. 
We will prove it in the Appendix.

\subsection{A Hadamard three-quasi-circle theorem}

Using the fundamental solution $\hat G$ for the operator $\hat L$, we can now prove the following.

\begin{thm}
Let $f$ be a function for which $\hat D f = 0$.
Set
$$M\pr{s} = \max\set{\abs{f\pr{z}} : z \in Z_s }.$$
Then for any $0 < s_1 < s_2 < s_3$,
\begin{equation}
\log\pr{\frac{s_3}{s_1}} \log M\pr{s_2} \le \log\pr{\frac{s_3}{s_2}} \log M\pr{s_1} + \log\pr{\frac{s_2}{s_1}} \log M\pr{s_3}.
\end{equation}
\label{Hadamard}
\end{thm}

\begin{proof}
Let $\mathcal{A}_{s_1, s_3} = \set{z : s_1 \le \ell\pr{ z} \le s_3} = \overline{ Q_{s_3} \setminus Q_{s_1}}$, where $\ell$ is associated to $\hat G$, the fundamental solution of $\hat L$.  
By Lemma \ref{ZsBounds}, this set is contained in an annulus with inner and outer radius depending on $s_1$, $s_3$, and $\la$.
In particular, it is bounded and does not contain the origin.
Therefore, $\hat G\pr{z}$ is bounded on $\mathcal{A}_{s_1, s_3}$. 
Let $z_0$ be in the interior of $\mathcal{A}_{s_1,s_3}$. 
If $f\pr{z_0}=0$, then $a \hat G\pr{z_0}+\log\abs{f\pr{z_0}}=-\infty$ for any $a\in\R$. 
On the other hand, if $f\pr{z_0} \ne 0$, then Lemma~\ref{logLem} implies that $\hat L \brac{a \hat G\pr{z} + \log\abs{f\pr{z}}} = 0$ for $z$ near $z_0$. 
By the maximum principle, $z_0$ cannot be an extremal point. 
Therefore, $a \hat G\pr{z} + \log\abs{f\pr{z}}$ takes it maximum value on the boundary of $\mathcal{A}_{s_1, s_3}$. We will choose the constant $a\in\R$ so that 
$$\max\set{a \hat G\pr{z} + \log\abs{f\pr{z}} : z \in Z_{s_1}}  = \max\set{a \hat G\pr{z} + \log\abs{f\pr{z}} : z \in  Z_{s_3}},$$
or rather
$$\log \pr{s_1^a M\pr{s_1}}  = \log\pr{s_3^a M\pr{s_3}}.$$
It follows that for any $z \in \mathcal{A}_{s_1, s_3}$,
\begin{align*}
a \hat G\pr{z} + \log\abs{f\pr{z}} &\le\log \pr{s_1^a M\pr{s_1}}  ({\mbox{or}}\; \log \pr{s_3^a M\pr{s_3}}).
%&\le\log \pr{s_i^a M\pr{s_i}}  \;{\mbox{for}}\; i=1,3.
\end{align*}
Furthermore, for any $s_2 \in \pr{s_1, s_3}$,
\begin{align*}
\max\set{a \hat G\pr{z} + \log\abs{f\pr{z}}: z \in Z_{s_2}} &\le \log \pr{s_1^a M\pr{s_1}} ({\mbox{or}}\;\log\pr{s_3^a M\pr{s_3}}),
%&\le \log \pr{s_i^a M\pr{s_i}} \; {\mbox{for}}\; i = 1,3,
\end{align*}
or
$$\log \pr{s_2^a M\pr{s_2}} \le \log \pr{s_1^a M\pr{s_1}} ({\mbox{or}}\; \log\pr{s_3^a M\pr{s_3}}).$$
%\le \log \pr{s_i^a M\pr{s_i}} \; {\mbox{for}}\; i=1,3
Consequently,
$${s_2^a M\pr{s_2}} \le {s_1^a M\pr{s_1}} \;({\mbox{or}}\;{s_3^a M\pr{s_3}}),$$
% {s_i^a M\pr{s_i}} \;{\mbox{for}}\;i=1,3,
so that for any $\tau \in \pr{0,1}$, since ${s_1^a M\pr{s_1}} = {s_3^a M\pr{s_3}}$, then
\begin{align*}
&{s_2^a M\pr{s_2}} \le  \brac{s_1^a M\pr{s_1}}^\tau \brac{s_3^a M\pr{s_3}}^{1-\tau} \\
& \brac{ M\pr{s_2}}^{\log\pr{\frac{s_3}{s_1}}} \le  \brac{\pr{\frac{s_1}{s_2}}^a M\pr{s_1}}^{\tau \log\pr{\frac{s_3}{s_1}}} \brac{\pr{\frac{s_3}{s_2}}^a M\pr{s_3}}^{\pr{1-\tau}\log\pr{\frac{s_3}{s_1}}}.
\end{align*}
We choose $\tau$ so that $\tau \log\pr{\frac{s_3}{s_1}} = \log\pr{\frac{s_3}{s_2}}$. 
Then $\pr{1-\tau} \log\pr{\frac{s_3}{s_1}} = \log\pr{\frac{s_2}{s_1}}$ and
\begin{align*}
&\brac{\pr{\frac{s_1}{s_2}}^a }^{\tau \log\pr{\frac{s_3}{s_1}}} \brac{\pr{\frac{s_3}{s_2}}^a }^{\pr{1-\tau}\log\pr{\frac{s_3}{s_1}}} 
= \exp\brac{a\log\pr{\frac{s_3}{s_2}} \log\pr{\frac{s_1}{s_2}} + a \log\pr{\frac{s_2}{s_1}} \log \pr{\frac{s_3}{s_2}} } = 1.
\end{align*}
Therefore,
\begin{align*}
& { M\pr{s_2}}^{\log\pr{\frac{s_3}{s_1}}} \le  { M\pr{s_1}}^{ \log\pr{\frac{s_3}{s_2}}} {M\pr{s_3}}^{\log\pr{\frac{s_2}{s_1}}}.
\end{align*}
Taking logarithms completes the proof.
\end{proof}

\begin{cor}\label{3circle}
Let $f$ satisfy $\hat Df=0$. Then for $0<s_1<s_2<s_3$
\[
\norm{f}_{L^\iny\pr{Q_{s_2}}} \le \pr{\norm{f}_{L^\iny\pr{Q_{s_1}}}}^\te \pr{\norm{f}_{L^\iny\pr{Q_{s_3}}}}^{1 - \te},
\]
where
\[
\te=\frac{\log(s_3/s_2)}{\log(s_3/s_1)}.
\]
\end{cor}
\begin{rem}
{From Remark~\ref{rem4.1}, we know that if $Df=0$, then $D_ff=0$. Hence Corollary~\ref{3circle} applies to such $f$}.
\end{rem}

\subsection{The similarity principle}

The approach here is based on the work of Bojarksi, as presented in \cite{Boj09}. 
We will start with a few definitions and facts that will be used below. 
For simplicity, we work on a bounded domain $\Om$. Define the operators
$$T{\om}\pr{z} = -\frac{1}{\pi} \iint_\Om \frac{\om\pr{\zeta}}{\zeta - z} d\zeta$$
$$S \om \pr{z} = -\frac{1}{\pi} \iint_\Om \frac{\om\pr{\zeta}}{\pr{\zeta - z}^2} d\zeta.$$

We will make use the of the following results, collected from \cite{Boj09}.

\begin{lem}
Suppose that $g \in L^p$ for some $p > 2$.  
Then $T g$ exists everywhere as an absolutely convergent integral and $S g$ exists almost everywhere as a Cauchy principal limit.
The following relations hold:
\begin{align*}
& \bar{\del} \pr{T g} = g \\
& \del \pr{T g} = S g \\
& \abs{T g\pr{z}} \le c_p \norm{g}_{L^p}  \\
& \norm{S g}_{L^p} \le C_p \norm{g}_{L^p} \\
& \lim_{p \to 2^+} C_p = 1.
%&  \norm{T g}_{L^2} \le \norm{g}_{L^2} \\
\end{align*}
\end{lem}

\begin{lem}[see Lemmas 4.1, 4.3 \cite{Boj09}]
Let $w$ be a generalized solution (possibly admitting isolated singularities) to
\begin{equation*}
\bar \del w + q_1\pr{z} \del w + q_2\pr{z} \ol{\del w} = A\pr{z} w + B\pr{z} \bar w
\end{equation*}
in a bounded domain $\Omega \su \R^2$.
Assume that $\abs{q_1\pr{z}} + \abs{q_2\pr{z}} \le  \alpha_0 < 1$ in $\Om$, and $A$, $B$ are bounded functions.
Then $w\pr{z}$ is given by
$$w\pr{z} = f\pr{z} e^{{T \om\pr{z}}} = f\pr{z} e^{\phi\pr{z}},$$
where $f$ is a solution to
$$\bar \del f + q_0\pr{z} \del f = 0$$
and
$$\phi\pr{z} = T{\om}\pr{z}.$$
Here, $\om$ solves  \eqref{intEq} and $q_0$ is defined by \eqref{q0}.
\label{simPrinc}
\end{lem}

The proof ideas are available in \cite{Boj09}.
For completeness, we include the proof.

\begin{proof}
Let $w\pr{z}$ be the generalized solution.
Set
$$h\pr{z} = \left\{\begin{array}{ll} 
A\pr{z} + B\pr{z} \frac{\bar w}{w} & \text{ for } w\pr{z} \ne 0 \text{ and } w\pr{z} \ne \iny \\  
A\pr{z} + B\pr{z} & \text{ otherwise }
\end{array}\right.$$
\begin{equation}\label{q0}
q_0\pr{z} = \left\{\begin{array}{ll} 
q_1\pr{z} + q_2\pr{z} \frac{\ol{\del w}}{\del w} & \text{ for } \del w \ne 0  \\  
q_1\pr{z} + q_2\pr{z} & \text{ otherwise }.
\end{array} \right.
\end{equation}
We have $\abs{q_0\pr{z}} \le \abs{q_1\pr{z}} + \abs{q_2\pr{z}} \le \alpha_0$. Consider the integral equation
\begin{equation}
\om + q_0 S \om = h.
\label{intEq}
\end{equation}
Let $p > 2$ be such that $C_p q_0 < 1$.
Since $h\pr{z} \in L^p\pr{\Omega}$, then by a fixed point argument, this integral equation has a unique solution $\om\pr{z} \in L^p$.
Set $\phi\pr{z} = T \om\pr{z}$, then define $f\pr{z} = w\pr{z} e^{- \phi\pr{z}}$.
We see that
\begin{align*}
\bar \del f &= \bar \del w e^{- \phi} - \bar \del \phi w e^{-\phi} = \bar \del w e^{- \phi} - \om w e^{-\phi} \\
\del f &= \del w e^{- \phi} - \del \phi w e^{-\phi} = \del w e^{- \phi} - S \om w e^{-\phi}.
\end{align*}
It follows that
\begin{align*}
\bar \del f + q_0 \del f &= \brac{\bar \del w + q_0 \del w - \pr{\om + q_0 S \om }w} e^{-\phi} \\
&= \brac{\bar \del w + q_0 \del w - hw} e^{-\phi} \\
&= \brac{\bar \del w + q_1 \del w + q_2 \ol{\del w} - Aw - B \bar w} e^{-\phi} \\
&= 0.
\end{align*}
\end{proof}

\begin{cor}
Let $w$ be a generalized solution (possibly admitting isolated singularities) to
\begin{equation*}
\bar \del w + q_1\pr{z} \del w + q_2\pr{z} \ol{\del w} = A\pr{z} w + B\pr{z} \bar w
\end{equation*}
in a bounded domain $\Omega \su \R^2$.
Assume that $\abs{q_1\pr{z}} + \abs{q_2\pr{z}} \le \alpha_0 < 1$ in $\Om$, and $A$, $B$ are bounded functions.
Then $w\pr{z}$ is given by
$$w\pr{z} = f\pr{z} g\pr{z},$$
where $f$ is a solution to
$$\bar \del f + q_0\pr{z} \del f = 0$$
and
$$\exp\brac{-C \pr{ \norm{A}_{L^\iny\pr{\Omega}} +  \norm{B}_{L^\iny\pr{\Omega}}}} \le \abs{g\pr{z}} \le \exp\brac{C \pr{ \norm{A}_{L^\iny\pr{\Omega}} +  \norm{B}_{L^\iny\pr{\Omega}}}}.$$
\label{simCor}
\end{cor}

\begin{proof}
From the previous lemma, we have that $g\pr{z} = \exp\pr{T \om \pr{z}}$, where $\om$ is the unique solution to \eqref{intEq}.
Since $C_p \alpha_0 < 1$, then
$$\norm{\om}_{L^p} \le C \norm{h}_{L^p}.$$
Therefore,
$$\abs{T \om \pr{z}} \le C \norm{h}_{L^p} \le C \brac{\norm{A}_{L^\iny\pr{\Omega}} +  \norm{B}_{L^\iny\pr{\Omega}}},$$
where $C$ depends on $\Omega$.
The conclusion follows.
\end{proof}

\section{The proof of Theorem \ref{localEst}}
\label{localEstProof}

Before we can prove Theorem \ref{localEst}, we need to develop a set of results that are analogous to those in \cite{KSW15}.
The first step is to show that a positive multiplier exists.
We then use this positive multiplier to transform the PDE \eqref{localPDE} into a divergence-form equation.
The divergence-form equation is used to introduce a stream function, which gives rise to a Beltrami equation.
Then, using the similarity principle of Bojarski and the Hadamard three-quasi-circle theorem, we are able to prove Theorem \ref{localEst}. 
From now on, unless specified otherwise, all constants $C$, $c$ depend on $\la$ and $\mu$.
Moreover, these constants are allowed to change from line to line.
We also use the more compact notation $\si\pr{\cdot}$ and $\rho\pr{\cdot}$ in place of $\si\pr{\cdot; \la}$ and $\rho\pr{\cdot; \la}$ where it is understood that these functions depend on the ellipticity constant $\la$.

The first step is to show that there exists a positive solution $\phi$ to \eqref{localPDE} in the ball $B_d$, where $d = \rho\pr{\frac{7}{5}} + \frac{2}{5}$.
Let $\eta$ be some constant to be determined and set 
$$\phi_1\pr{x,y} = \exp\pr{\eta x}.$$ 
Then by \eqref{VBd}, \eqref{aGrBd}, and \eqref{bndBelow}
\begin{align*}
\di \pr{A \gr \phi_1} - V \phi_1 
&= \brac{\eta \pr{\del_x a_{11} + \del_y a_{12}} + \eta^2 a_{11} - V} \phi_1 \\
&\ge \brac{\la \eta^2  - M - 2 \eta \mu \sqrt{M}} \phi_1.
\end{align*}
If $\eta = c_1 \sqrt{M}$ for some constant $c_1$ depending on $\la$ and $\mu$ that is sufficiently large, then $\phi_1$ is a subsolution.  
Now define $\phi_2 = \exp\pr{c_2 \sqrt{M}}$, where $c_2$ is a constant chosen so that $\phi_2 \ge \phi_1$ on $B_d$.
Since $V \ge 0$, then $L \phi_2 - V \phi_2 \le 0$, so $\phi_2$ is a supersolution.  
It follows that there exists a positive solution $\phi$ to \eqref{localPDE} such that
\begin{equation}
\exp\pr{-C_1 \sqrt{M}} \le \phi\pr{z} \le \exp\pr{C_1 \sqrt{M}} \;\;\;\ \text{ for all } z \in B_d,
\label{phiBd}
\end{equation}
where $C_1$ depends on $c_1, c_2$, and $\la$.

Furthermore, (by Theorem 8.32 from \cite{GT01}, for example) for $0 < \al < 1$, $s< 2$,  $\phi$ satisfies the interior estimate
\begin{equation}
\norm{\gr \phi}_{L^\iny\pr{Q_{\al s}}} \le C \norm{\phi}_{L^\iny\pr{Q_{s}}},
\label{intEst}
\end{equation}
where $C = C\pr{\la, K, M, s, \al}$, with
\begin{align*}
K &= \max_{i,j = 1, 2} \abs{a_{ij}}_{0, \gamma; Q_2} 
= \max_{i,j = 1, 2} \brac{\norm{a_{ij}}_{L^\iny\pr{Q_2}} + \sup_{x \ne y \in Q_2} \frac{\abs{a_{ij}\pr{x} - a_{ij}\pr{y}}}{\abs{x-y}^\gamma}},
\end{align*}
where $0 < \gamma < 1$ is arbitrary.
Note that since
\begin{align*}
\sup_{x \ne y \in Q_2} \frac{\abs{a_{ij}\pr{x} - a_{ij}\pr{y}}}{\abs{x-y}^\gamma} 
&\le \sup_{x \ne y \in Q_2} \mu \frac{\abs{x-y}}{\abs{x-y}^\gamma} 
= \sup_{x \ne y \in Q_2} \mu \abs{x-y}^{1-\gamma}
\le \mu \diam\pr{Q_2},
\end{align*}
then $C = C\pr{ \la, \mu, M, s, \al}$.
Moreover, by scaling considerations and Lemma \ref{ZsBounds},
$$C\pr{\la, \mu, M, s, \al} \le \frac{C\pr{\la, \mu, M, \al}}{s^c}.$$

Set $v = u/ \phi$.
Since $u$ and $\phi$ are both solutions to \eqref{localPDE} and $A$ is symmetric by \eqref{symm}, we see that
\begin{equation}
\di \pr{ \phi^2 A \gr v} = 0 \, \text{ in } B_d \subset  \R^2.
\label{grPDE}
\end{equation}
We use \eqref{grPDE} to define a stream function in $B_d$.  
Let $\tilde v$, with $\tilde v\pr{0} = 0$, satisfy the following system of equations
\begin{equation}
\left\{\begin{array}{rl}
\tilde{v}_y &= \phi^2 \pr{a_{11} v_x + a_{12}  v_y } \\
-\tilde{v}_x &= \phi^2 \pr{a_{12} v_x + a_{22} v_y } .
\end{array}  \right.
\label{streamFunc}
\end{equation}
Specifically, for any $\pr{x,y} \in B_d$,
\begin{equation}
\tilde v\pr{x,y} = \int_0^y \phi^2 \pr{a_{11} v_x + a_{12} v_y }\pr{0,t} dt - \int_{0}^ x \phi^2 \pr{a_{12} v_x + a_{22} v_y }\pr{s,y} ds.
\label{tildev}
\end{equation}
The stream function is used to transform the divergence-free equation into a Beltrami equation.
Set $w = \phi^2 v + i \tilde{v}$.  
Then, using \eqref{streamFunc}, we see that
\begin{align*}
D w
&= 2\phi D \phi v + \phi^2 D v + D(i \tilde{v})\\
&= 2 D\pr{\log \phi} \phi^2 v \\
&+ \phi^2 \brac{\frac{\pr{a_{11} + \det A} + i a_{21}}{\det\pr{A+I}} v_x +\frac{a _{12} + i\pr{a_{22} + \det A}}{\det\pr{A+I}} v_y } 
+ i \brac{ \frac{\pr{a_{11} +1} + i a_{12}}{\det\pr{A+I}} \tilde{v}_x + \frac{a_{12} + i\pr{a_{22} + 1}}{\det\pr{A+I}} \tilde{v}_y } \\
&= 2 D\pr{\log \phi} \phi^2 v \\
&+ \phi^2 \brac{\frac{\pr{a_{11} + \det A} + i a_{21}}{\det\pr{A+I}} v_x +\frac{a _{12} + i\pr{a_{22} + \det A}}{\det\pr{A+I}} v_y } \\
&+ i \phi^2 \brac{ -\frac{\pr{a_{11} +1} + i a_{12}}{\det\pr{A+I}} \pr{a_{12} v_x + a_{22} v_y } + \frac{a_{12} + i\pr{a_{22} + 1}}{\det\pr{A+I}} \pr{ a_{11} v_x + a_{12} v_y }}  \\
&= D\pr{\log \phi} \pr{w + \ol{w}}.
\end{align*}
Therefore,
\begin{align}
D w %&= 2\phi \widehat{\bar \del }\phi v \\
&= \al \pr{w + \bar w},
\label{diffEq}
\end{align}
where $\al = D \pr{ \log \phi}$.

The next step is to estimate $\al$.
Here we mimic the arguments from \cite{KSW15}, making appropriate modifications to account for the variable coefficients of the operator.
To understand the behavior of $\al$, we will study $\psi = \log \phi$.  
From \eqref{phiBd}, we see that
\begin{equation}
\abs{\psi\pr{z}} \le C \sqrt{M} \,\,\, \text{ in } B_d.
\label{psiBd}
\end{equation}
\nid
Furthermore, a computation shows that $\psi$ solves the following equation
\begin{equation}
\di \pr{A \gr \psi} + A \gr \psi \cdot \gr \psi = V \,\,\, \text{ in } B_d.
\label{psiPDE}
\end{equation}

\begin{lem}
If $\phi$ is a solution to \eqref{localPDE} and $\psi = \log \phi$, then
\begin{equation}
\norm{\gr \psi}_{L^\iny\pr{B_{\rho\pr{7/5}}}} \le C \sqrt{M},
\label{psiGradBd}
\end{equation}
where $C$ depends on $\la, \mu$.
\label{LinyBd}
\end{lem}

\begin{proof}
Recall that $d = \rho\pr{7/5} + 2/5$.
Let $\te \in C^\iny_0\pr{B_d}$ be a cutoff function for which $\te \equiv 1$ in $B_{\rho\pr{7/5}+ 1/5}$.
Multiply \eqref{psiPDE} by $\te$ and integrate by parts:
\begin{align*}
\la \int \te \abs{\gr \psi}^2 
\le \int \te A \gr \psi \cdot \gr \psi 
= \int \te V - \int \di\pr{A \gr \te}  \psi
\le C\pr{M + \sqrt{M}}.
\end{align*}
It follows that
$$\int_{B_{\rho\pr{7/5}+1/5}} \abs{\gr \psi}^2 \le C M.$$

We rescale equation \eqref{psiPDE}.
Set $\vp = \frac{\psi}{C \sqrt{M}}$ for some $C > 0$.
Then \eqref{psiPDE} is equivalent to 
\begin{equation}
\eps \di \pr{A \gr \vp} + A \gr \vp \cdot \gr \vp = \tilde V \,\,\, \text{ in } B_d,
\label{vpPDE}
\end{equation}
where $\eps = \frac{1}{C \sqrt{M}}$ and $\tilde V = \frac{V}{C^2 M}$.
Now choose $C$ sufficiently large so that
\begin{align}
\norm{\tilde V}_{L^\iny\pr{B_d}} \le 1, \quad 
\norm{\vp}_{L^\iny\pr{B_d}} \le 1, \quad 
\int_{B_{\rho\pr{7/5}+1/5}} \abs{\gr \vp}^2 \le 1.
\label{scaleBds}
\end{align}

\begin{clm}
Let $c> 0$ be such that for any $z \in B_{\rho\pr{7/5}}$, $B_{2c/5}\pr{z} \su B_{\rho\pr{7/5}+1/5}$.
For any $z \in B_{\rho\pr{7/5}}$, and $\eps < r < c/5$, if \eqref{vpPDE} and \eqref{scaleBds} hold, then
$$\int_{B_r\pr{z}} \abs{\gr \vp}^2 \le C r^2.$$
\label{clm1}
\end{clm}

\begin{proof}[Proof of Claim \ref{clm1}]
It suffices to take $z = 0$.
Let $\eta \in C^\iny_0\pr{B_{2r}}$ be a cutoff function such that $\eta \equiv 1$ in $B_r$.
Set $\disp m = \abs{B_{2r}}^{-1} \int_{B_{2r}} \vp$.
By the divergence theorem
\begin{align}
0 &= \eps \int \di\pr{A \gr\brac{\pr{\vp - m} \eta^2}} \nonumber \\
&= \eps \int \di \pr{A \gr \vp} \eta^2 + 4 \eps \int \eta A \gr \vp \cdot \gr \eta + \eps \int \pr{\vp - m} \di \pr{A \gr\pr{\eta^2}}.
\label{DivThmRes}
\end{align}
We now estimate each of the three terms.
By \eqref{vpPDE} and \eqref{scaleBds},
\begin{align}
\int \eps \di \pr{A \gr \vp} \eta^2 
&= - \int A \gr \vp \cdot \gr \vp \eta^2 + \int \tilde V \eta^2
\le - \la \int \abs{\gr \vp}^2 \eta^2 + \norm{\tilde V}_{L^\iny\pr{B_d}} \int_{B_{2r}} 1 \nonumber \\
&\le - \la \int \abs{\gr \vp}^2 \eta^2 + C r^2.
\label{IBd}
\end{align}
By Cauchy-Schwarz and Young's inequality,
\begin{align}
\abs{4 \eps \int \eta A \gr \vp \cdot \gr \eta} 
&\le 4 \la \eps \pr{ \int \abs{\gr \vp}^2 \eta^2 }^{1/2} \pr{ \int \abs{\gr \eta}^2 }^{1/2} 
\le \frac{\la}{2} \int \abs{\gr \vp}^2 \eta^2 + C \eps^2.
\label{IIBd}
\end{align}
For the third term, we use the Poincar\'e inequality to show that
\begin{align}
\abs{\eps \int \pr{\vp - m} \di \pr{A \gr\pr{\eta^2}}} 
&\le C \eps r^{-2} \int_{B_{2r}} \abs{\vp - m}
\le C \pr{\int_{B_{2r}} \abs{\vp - m}^2}^{1/2} \pr{\int_{B_{2r}} \eps^2 r^{-4}}^{1/2} \nonumber \\
&\le C r \pr{\int_{B_{2r}} \abs{\gr \vp}^2}^{1/2} \pr{ \eps^2 r^{-2}}^{1/2}
\le C \eps^2 + \frac{1}{400} \int_{B_{2r}} \abs{\gr \vp}^2.
\label{IIIBd}
\end{align}
Combining \eqref{DivThmRes}-\eqref{IIIBd}, we see that
\begin{align}
\int_{B_r} \abs{\gr \vp}^2 \le C \eps^2 + C r^2 + \frac{1}{200 \la} \int_{B_{2r}} \abs{\gr \vp}^2
\le C r^2 + \frac{1}{200 \la} \int_{B_{2r}} \abs{\gr \vp}^2.
\label{combinedEst}
\end{align}

If $r^2 \ge \frac{1}{200}$, then by the last estimate of \eqref{scaleBds}, the inequality above implies that
$$\int_{B_r} \abs{\gr \vp}^2 \le C r^2.$$
Otherwise, if $r^2 < \frac{1}{200}$, choose $k \in \N$ so that
$$\frac{c}{5} \le 2^k r \le \frac{2c}{5}.$$
Clearly, $r^2 \ge C \pr{1/200 \la}^k$.
It follows from repeatedly applying \eqref{combinedEst} that
\begin{align*}
\int_{B_r} \abs{\gr \vp}^2
\le C r^2 + \pr{\frac{1}{200 \la}}^k \int_{B_{2^kr}} \abs{\gr \vp}^2
\le C r^2,
\end{align*}
proving the claim.
\end{proof}

We now use Claim \ref{clm1} to give a pointwise bound for $\gr \vp$ in $B_{\rho\pr{7/5}}$.
Define 
$$\vp_\eps\pr{z} = \frac{1}{\eps} \vp\pr{\eps z}, \;\; A_\eps\pr{z} = A\pr{\eps z}, \;\; L_\eps = \di A_\eps \gr.$$
Then
$$\gr \vp_\eps\pr{z} = \gr \vp \pr{\eps z}, \qquad  L_\eps \vp_\eps\pr{z} = \eps L \vp \pr{\eps z}.$$
It follows from \eqref{vpPDE} that
\begin{equation*}
L_\eps \vp_\eps + A_\eps \gr \vp_\eps \cdot \gr \vp_\eps = \tilde V\pr{\eps z} : = \tilde V_\eps\pr{z},
%\label{vpepsPDE}
\end{equation*}
$$\norm{\tilde V_\eps}_{L^\iny\pr{B_1}} \le 1.$$
Moreover,
\begin{align*}
%\int_{Q_2} \abs{\gr \vp_\eps}^2 \le 
\int_{B_{2}}\abs{\gr \vp\pr{\eps z}}^2 
= \frac{1}{\eps^2} \int_{B_{2 \eps}} \abs{\gr \vp}^2
\le \frac{1}{\eps^2} C \eps^2 = C,
\end{align*}
where we have used Claim \ref{clm1}.
It follows from Theorem 2.3 and Proposition 2.1 in Chapter V of \cite{Gi83} that there exists $p > 2$ such that
\begin{equation}
\norm{\gr \vp_\eps}_{L^p\pr{B_1}} \le C.
\label{QiaBd}
\end{equation}
Now we define
$$\tilde \vp_\eps \pr{z} = \vp_\eps\pr{z} - \frac{1}{\abs{B_1}} \int_{B_1} \vp_\eps.$$
Since $\gr \tilde \vp_\eps = \gr \vp_\eps$, then 
$$L_\eps \tilde \vp_\eps = - A_\eps \gr \tilde \vp_\eps \cdot \gr \tilde \vp_\eps + \tilde V_\eps : = \zeta \quad \text{ in } B_1.$$
Clearly, $\norm{\zeta}_{L^{p/2}\pr{B_1}} \le C$.
Moreover, by H\"older, Poincar\'e and \eqref{QiaBd},
$$\norm{\tilde \vp_\eps}_{L^{p/2}\pr{B_1}} 
\le C \norm{\tilde \vp_\eps}_{L^{p}\pr{B_1}} 
\le C \norm{\gr \tilde \vp_\eps}_{L^p\pr{B_1}} \le C.$$
By Theorem 9.11 from \cite{GT01}, for every $\eps < \eps_0$,
$$\norm{\tilde \vp_\eps}_{W^{2, p/2}\pr{B_r}} \le C,$$
for any $r < 1$, where $C$ depends on $\eps_0$ and $r$.
By repeating these arguments, we obtain that
$$\norm{\gr \tilde \vp_\eps}_{L^\iny\pr{B_{r^\prime}}} = \norm{\gr \vp_\eps}_{L^\iny\pr{B_{r^\prime}}} = \norm{\gr \vp}_{L^\iny\pr{B_{ \eps r^\prime}}} \le C,$$
for $r^\prime < r$.
This derivation works for any $z \in B_{\rho\pr{7/5}}$ and any $\eps < \eps_0$. 
Since $\vp = \frac{\psi}{C \sqrt{M}}$, conclusion of the lemma follows.
\end{proof}

Using that the coefficients of $D$ are bounded, we obtain the following corollary.

\begin{cor}
If $\al = D \psi$ and $\norm{\gr \psi}_{L^\iny\pr{B_{\rho\pr{7/5}}}} \le C \sqrt{M}$, then
\begin{equation}
\norm{\al}_{L^\iny\pr{B_{\rho\pr{7/5}}}} \le C \sqrt{M}.
\label{alBd}
\end{equation}
\label{alCor}
\end{cor}

We have now have all the tools we need to prove Theorem \ref{localEst}.

\begin{proof}
As shown using the stream function \eqref{streamFunc}, if $u$ is a solution to \eqref{localPDE} in $B_d$, then $w = \phi^2 v + i \tilde v$ is a solution to \eqref{diffEq} in $B_d \Supset B_{\rho\pr{7/5}}$.
By the similarity principle given in Lemma \ref{simPrinc} and Corollary \ref{simCor}, any solution to \eqref{diffEq} in $B_{\rho\pr{7/5}}$ is a function of the form
$$w\pr{z} = f\pr{z} g\pr{z},$$
with
$$D_w f = 0 \;\; \text{ in } \; B_{\rho\pr{7/5}}$$
and
\begin{align}
\exp\pr{-C \sqrt{M} } \le & \abs{g\pr{z}} \le \exp\pr{C \sqrt{M}} \;\; \text{ in } B_{\rho\pr{7/5}},
\label{gBnd}
\end{align}
where we have used \eqref{alBd} and the definition of $D_w$ is given in \eqref{DwDef}. 
By Corollary \ref{3circle}, the Hadamard three-quasi-circle theorem applied to $D_w$, 
\begin{align}
\norm{f}_{L^\iny\pr{Q_{s_1}}} \le \pr{\norm{f}_{L^\iny\pr{Q_{s/2}}}}^\te \pr{\norm{f}_{L^\iny\pr{Q_{s_2}}}}^{1 - \te},
\label{3balls}
\end{align}
where $\frac{s}{2} < s_1 < s_2 < \frac{7}{5}$ and 
$$\te = \frac{\log\pr{s_2/s_1}}{\log\pr{2s_2/s}}.$$
Substituting $f = w g^{-1}$ into \eqref{3balls} and using \eqref{gBnd}, we see that
\begin{align}
\norm{ w}_{L^\iny\pr{Q_{s_1}}} \le \exp\pr{C \sqrt{M}} \pr{\norm{ w}_{L^\iny\pr{Q_{s/2}}}}^\te \pr{\norm{ w}_{L^\iny\pr{Q_{s_2}}}}^{1 - \te}.
\label{3balls2}
\end{align}
Since $w = \phi^2 v + i \tilde v = \phi u + i \tilde v$, then
\begin{align*}
\abs{\phi u} \le \abs{w} \le \abs{\phi u} +\abs{\tilde v}.
\end{align*}
It follows from expression \eqref{tildev} and Lemma \ref{ZsBounds} that for any $z \in Q_s$, where $s < 2$,
\begin{align*}
\abs{\tilde v\pr{z}} \le C s^c \exp\pr{2 C_1 \sqrt{M}} \norm{\gr v}_{L^\iny\pr{Q_s}} .
\end{align*}
Using that $v = u/\phi$, the bounds for $\phi$ in \eqref{phiBd}, and the interior estimate \eqref{intEst}, we see that
%\begin{align*}
%\abs{f\pr{z}} \le C \exp\pr{C \sqrt{M}} \norm{u}_{L^\iny\pr{B_{2r}}} .
%\end{align*}
%If we assume \eqref{uBd},
%$$\norm{u}_{L^\iny\pr{Q_2}} \le \exp\pr{C_0 \sqrt{M}}, $$
%for some $C_0 > 0$, then 
setting $s_1 = 1$ and $s_2 = \frac{6}{5}$ in \eqref{3balls2} gives
\begin{align*}
\norm{u}_{L^\iny\pr{Q_{1}}}
%&\le \exp\pr{2 c_1 \sqrt{M}} \norm{ \phi u}_{L^\iny\pr{Q_{1}}} \\
%&\le \exp\pr{2 c_1 \sqrt{M}} \pr{\norm{ f}_{L^\iny\pr{Q_{s/2}}}}^\te \pr{\norm{ f}_{L^\iny\pr{Q_{s_2}}}}^{1 - \te} \\
&\le \exp\pr{C \sqrt{M}} \pr{ s^{-c} \norm{ u}_{L^\iny\pr{Q_{s}}}}^\te \pr{\norm{ u}_{L^\iny\pr{Q_{7/5}}}}^{1 - \te} \\
&\le \exp\pr{C \sqrt{M}} \pr{s^{-c} \norm{ u}_{L^\iny\pr{Q_{s}}}}^\te,
\end{align*}
where we have applied \eqref{uBd} to the last term on the right.
Since $\norm{u}_{L^\iny\pr{Q_1}} \ge \norm{u}_{L^\iny\pr{B_b}} \ge 1$, we have
\begin{equation*}
\norm{ u}_{L^\iny\pr{Q_{s}}} \ge s^{c}\exp\pr{-\frac{C \sqrt{M}}{\te}}.
\end{equation*}
By Lemma \ref{ZsBounds}, $Q_s \su B_{s^{c_1}}$.
It follows that
\begin{align*}
\norm{ u}_{L^\iny\pr{B_{r}}} \ge r^{C \sqrt{M}},
\end{align*}
as claimed
\end{proof}

\section{The proof of Theorem \ref{localEstgrW}}
\label{localEstgrWProof}

By building on the techniques from the previous section, we show here how to prove Theorem \ref{localEstgrW}. 
To transform equation \eqref{localPDEgrW} into a divergence-free equation, we will construct a positive solution to the adjoint equation.
That is, we show there exists a positive solution $\phi$ to \eqref{localPDEngW}.

Let $\eta$ be some constant to be determined and set 
$$\phi_1\pr{x,y} = \exp\pr{\eta x}.$$ 
Then by \eqref{VBd},  \eqref{WBd}, \eqref{aGrBd}, and \eqref{bndBelow}
\begin{align*}
\di \pr{A \gr \phi_1} - W \cdot \gr \phi_1 - V \phi_1 
&= \brac{\eta \pr{\del_x a_{11} + \del_y a_{12}} + \eta^2 a_{11} - \eta W_1 - V} \phi_1 \\
&\ge \brac{ \la \eta^2  - M - 2 \eta \mu \sqrt{M} - \eta \sqrt{M}} \phi_1.
\end{align*}
If $\eta = c_1 \sqrt{M}$ for some constant $c_1$ depending on $\la$ and $\mu$ that is sufficiently large, then $\phi_1$ is a subsolution.  
Now define $\phi_2 = \exp\pr{c_2 \sqrt{M}}$, where $c_2$ is a constant chosen so that $\phi_2 \ge \phi_1$ on $B_d$.
Since $V \ge 0$, then $L \phi_2 - W \cdot \gr \phi_2 - V \phi_2 \le 0$, so $\phi_2$ is a supersolution.  
It follows that there exists a positive solution $\phi$ to \eqref{localPDEngW} such that \eqref{phiBd} holds.
As above, a version of the interior estimate \eqref{intEst} holds for $\phi$.
Set $v = u/ \phi$.
Using the equations for $u$ and $\phi$, and that $A$ is symmetric, we see that
\begin{equation}
\di \pr{ \phi^2 A \gr v + \phi^2 W v} = 0 \text{ in } B_d \subset  \R^2.
\label{grPDEgrW}
\end{equation}

Since $\phi^2 A \gr v + \phi^2 W v$ is divergence-free, then there exists $\tilde v$, with $\tilde v\pr{0} = 0$, for which
\begin{equation}
\left\{\begin{array}{rl}
\del_y \tilde{v} &= \phi^2 \pr{a_{11} v_x + a_{12} v_y + W_1 v } \\
-\del_x \tilde{v} &= \phi^2 \pr{a_{12} v_x + a_{22} v_y + W_2 v} .
\end{array}  \right.
\label{streamFuncgrW}
\end{equation}
That is, for any $\pr{x,y} \in B_d$,
\begin{equation}
\tilde v\pr{x,y} = \int_0^y \phi^2 \pr{a_{11} v_x + a_{12} v_y + W_1 v }\pr{0,t} dt - \int_{0}^ x \phi^2 \pr{a_{12} v_x + a_{22} v_y + W_2 v}\pr{s,y} ds.
\label{tildev2}
\end{equation}
Set $w = \phi^2 v + i \tilde{v}$.  
Then, using \eqref{streamFuncgrW}, %, by the stream function definition above, we see that
\begin{align*}
D w 
&= 2\phi D \phi v + \phi^2 D v + D(i\tilde{v}) \\
&= 2 D\pr{\log \phi} \phi^2 v \\
&+ \phi^2 \brac{\frac{\pr{a_{11} + \det A} + i a_{21}}{\det\pr{A+I}} v_x +\frac{a _{12} + i\pr{a_{22} + \det A}}{\det\pr{A+I}} v_y } 
+ i \brac{ \frac{\pr{a_{11} +1} + i a_{12}}{\det\pr{A+I}} \tilde{v}_x + \frac{a_{12} + i\pr{a_{22} + 1}}{\det\pr{A+I}} \tilde{v}_y } \\
&= 2 D\pr{\log \phi} \phi^2 v \\
&+ \phi^2 \brac{\frac{\pr{a_{11} + \det A} + i a_{21}}{\det\pr{A+I}} v_x +\frac{a _{12} + i\pr{a_{22} + \det A}}{\det\pr{A+I}} v_y } \\
&+ i \phi^2 \brac{ -\frac{\pr{a_{11} +1} + i a_{12}}{\det\pr{A+I}} \pr{a_{12} v_x + a_{22} v_y + W_2 v} + \frac{a_{12} + i\pr{a_{22} + 1}}{\det\pr{A+I}} \pr{ a_{11} v_x + a_{12} v_y + W_1 v}}  \\
&= 2\brac{D\pr{\log \phi} + \frac{a_{12} - i\pr{a_{11} +1}}{2 \det\pr{A+I}} W_2 + \frac{-\pr{a_{22} + 1} + ia_{12}}{2 \det\pr{A+I}} W_1 } \phi^2 v.
\end{align*}
Therefore,
\begin{align}
D w %&= 2\phi \widehat{\bar \del }\phi v \\
&= \be \pr{w + \bar w}
\label{diffEq2}
\end{align}
where 
\begin{align*}
\be &= D\pr{\log \phi} + \frac{a_{12} - i\pr{a_{11} +1}}{2 \det\pr{A+I}} W_2 + \frac{-\pr{a_{22} + 1} + ia_{12}}{2 \det\pr{A+I}} W_1 \\
&:= \al + \frac{a_{12} - i\pr{a_{11} +1}}{2 \det\pr{A+I}} W_2 + \frac{-\pr{a_{22} + 1} + ia_{12}}{2 \det\pr{A+I}} W_1.
\end{align*}

\begin{lem}
If $\phi$ is a solution to \eqref{localPDEngW} and $\psi = \log \phi$, then
\begin{equation*}
\norm{\gr \psi}_{L^\iny\pr{B_{\rho\pr{7/5}}}} \le C \sqrt{M},
%\label{psiGradBd}
\end{equation*}
where $C$ depends on $\la, \mu$.
\label{LinyBdwW}
\end{lem}

The proof of Lemma \ref{LinyBdwW} is analagous to that of Lemma \ref{LinyBd}, except that we must include the magnetic potential $W$.
We omit the details since the arguments in \cite{KSW15} may be combined with the proof of Lemma \ref{LinyBd} above.
If we combine Lemma \ref{LinyBdwW}, Corollary \ref{alCor}, the bounds on $A$ from \eqref{bndBelow} and \eqref{bndAbove}, and the bounds on $W$ in \eqref{WBd}, we see that
\begin{align*}
\norm{\be}_{L^\iny\pr{B_{\rho\pr{7/5}}}} 
&\le \norm{\al}_{L^\iny\pr{B_{\rho\pr{7/5}}}} + C \norm{W_1}_{L^\iny\pr{B_{\rho\pr{7/5}}}} + C \norm{W_2}_{L^\iny\pr{B_{\rho\pr{7/5}}}}  \\
&\le C \sqrt{M}
\end{align*}
The proof of Theorem \ref{localEstgrW} follows that of Theorem \ref{localEst}, where we replace the bounds for $\al$ with the bounds for $\be$.

\section{The proof of Theorem \ref{localEstngW}}
\label{localEstngWProof}

The establish the local order-of-vanishing estimate for \eqref{localPDEngW}, we will use a trick similar to the one that appears in \cite{KSW15}.
Instead of transforming \eqref{localPDEngW} into a divergence-free equation and defining a stream function, we construct an equation of the form $D w = \widetilde W w$. Recall that $\det A = 1$ in \eqref{localPDEngW}.
Also, we now have an ellipticity constant of $\la^2$ in \eqref{ellip} instead of $\la$.
Therefore, in what follows, the shortened notations $\si\pr{\cdot}$ and $\rho\pr{\cdot}$ stand for $\si\pr{\cdot; \la^2}$ and $\rho\pr{\cdot; \la^2}$.
%Consequently, the bounds are still in terms of $\la$, but the dependence is different.

As shown in the previous section, there exists a positive solution $\phi$ to \eqref{localPDEngW} such that \eqref{phiBd} and \eqref{intEst} hold.
Set $v = u/ \phi$, where $u$ is a solution to \eqref{localPDEngW}.
It follows that
\begin{align}
L v + \pr{ 2 A \gr \psi - W } \cdot \gr v
&= \di \pr{A \frac{\gr u}{\phi} - A \frac{u \gr \phi}{\phi^2}} +  \pr{2 \frac{A \gr \phi}{\phi} - W}\cdot\pr{\frac{\gr u}{\phi} - \frac{u \gr \phi}{\phi^2}} =0,
\label{Lveqn}
\end{align}
where $\psi = \log \phi$.
Using the decomposition given by Lemma \ref{decompLem}, we may rewrite \eqref{Lveqn} as
\begin{equation}
D \widetilde D v = \widetilde W \widetilde D v + \pr{W - 2 A \gr \psi} \cdot \gr v.
\label{Deqn}
\end{equation}

\begin{lem}
There exists $\widetilde \Upsilon$ so that 
\begin{equation}
\pr{W - 2 A \gr \psi} \cdot \gr v =\widetilde \Upsilon \widetilde D v.
\label{UpsEqn}
\end{equation}
Moreover,
\begin{equation}
\norm{\widetilde \Upsilon}_{L^\iny\pr{B_{\rho\pr{7/5}}}} \le C \sqrt{M}.
\label{UpsBd}
\end{equation}
\end{lem}

\begin{proof}
Set $\Upsilon = e + i f$, where $e, f$ are real-valued functions to be determined.
Then
\begin{align*}
\Upsilon \widetilde D v 
&= \pr{e + i f} \set{\brac{1 + a_{11} - i a_{12}}\del_x v + \brac{a_{12} - i \pr{1 + a_{22}}}\del_y v} \\
&= \brac{e \pr{1 + a_{11}} + f a_{12}}\del_x v + \brac{e a_{12} + f \pr{1 + a_{22}}} \del_y v 
+ i \set{\brac{f\pr{1 + a_{11}} - e a_{12}} \del_x v + \brac{f a_{12} - e\pr{1 + a_{22}}} \del_y v}
\end{align*}
so that,
\begin{align*}
\frac{1}{2}\brac{\Upsilon \widetilde D v + \overline{\Upsilon  \widetilde D v} }
&= \brac{e \pr{1 + a_{11}} + f a_{12}}\del_x v + \brac{e a_{12} + f \pr{1 + a_{22}}} \del_y v.
\end{align*}
If we define 
$$\widetilde \Upsilon = \left\{\begin{array}{ll} \frac{1}{2}\brac{\Upsilon + \bar {\Upsilon}\frac{ \overline{ \widetilde D v}}{\widetilde D v}} & \text{ whenever }  \widetilde D v \ne 0 \\ 0 & \text{ otherwise } \end{array}\right.,$$
then \eqref{UpsEqn} will be satisfied if we choose $e$, $f$ so that
\begin{align*}
e \pr{1 + a_{11}} + f a_{12} &= W_1 - 2 a_{11} \frac{\del_x \phi}{\phi} - 2 a_{12} \frac{\del_y \phi}{\phi} \\
e a_{12} + f \pr{1 + a_{22}} &= W_2 - 2 a_{12}\frac{\del_x \phi}{\phi} - 2 a_{22} \frac{\del_y \phi}{\phi}.
\end{align*}
Solving this system, we see that
\begin{align*}
\brac{\begin{array}{l} e \\ f \end{array}} 
&= \frac{1}{\det\pr{A+I}} \brac{\begin{array}{ll} 1 + a_{22} & - a_{12} \\ - a_{12} & 1 + a_{11}  \end{array}}  \brac{\begin{array}{l} W_1 - 2 a_{11} \frac{\del_x \phi}{\phi} - 2 a_{12} \frac{\del_y \phi}{\phi} \\ W_2 - 2 a_{12}\frac{\del_x \phi}{\phi} - 2 a_{22} \frac{\del_y \phi}{\phi} \end{array}}  \\
&=  \frac{1}{\det\pr{A+I}} \brac{\begin{array}{l} \pr{1 + a_{22}} W_1 - a_{12} W_2 -2\pr{1 + a_{11}}\frac{\del_x \phi}{\phi} - 2 a_{12}  \frac{\del_y \phi}{\phi} \\
- a_{12}W_1 + \pr{1 + a_{11}} W_2  - 2 a_{12}  \frac{\del_x \phi}{\phi}  -2 \pr{1 + a_{22}}  \frac{\del_y \phi}{\phi}  \end{array}}.
\end{align*}
We may apply Lemma \ref{LinyBdwW}, with $\la$ replaced by $\la^2$ wherever necessary, to conclude that $\norm{\gr \psi}_{L^\iny\pr{B_{\rho\pr{7/5}}}} \le C \sqrt{M}$.
Combining this with the bounds on $A$ and $W$ leads to \eqref{UpsBd} and completes the proof.
\end{proof}

Returning to \eqref{Deqn}, we see that
$$D \widetilde D v = \pr{\widetilde W + \widetilde \Upsilon} \widetilde D v.$$
Since
\begin{align*}
\norm{\widetilde W + \widetilde \Upsilon}_{L^\iny\pr{B_{\rho\pr{7/5}}}} &\le C \sqrt{M},
\end{align*}
then we may apply the results from the previous section to the equation above, where $\widetilde D v$ now plays the role of $w$.
An application of the similarity principle, Lemma \ref{simPrinc} applied to $D$, shows that
$$\widetilde D v\pr{z} = f\pr{z} g\pr{z},$$
where $D f = 0$ in $B_{\rho\pr{7/5}}$ and $\exp\pr{- C \sqrt{M}} \le \abs{g\pr{z}} \le \exp\pr{C \sqrt{M}}$.
Then the Hadamard three-quasi-circle theorem (with respect to the operator $D$) is used as in the proof of Theorem \ref{localEst} above, along with $\|{\widetilde D v}\| \sim \|{\gr v}\|$, to show that
\begin{align}
\norm{ \gr v}_{L^\iny\pr{Q_{s_1}}} \le \exp\pr{C \sqrt{M}} \pr{\norm{ \gr v}_{L^\iny\pr{Q_{s/2}}}}^\te \pr{\norm{ \gr v}_{L^\iny\pr{Q_{s_2}}}}^{1 - \te},
\label{3balls2gr}
\end{align}
where $s/2 < s_1$, $s_1 = 6/5$, $s_2 = 13/10$ and
$$\te = \frac{\log\pr{s_2/s_1}}{\log\pr{2s_2/s}}.$$
Using the interior bound \eqref{intEst}, as well as the bound on $u$ given in \eqref{uBd}, we have
\begin{align}
\norm{ \gr v}_{L^\iny\pr{Q_{6/5}}} \le \exp\pr{C \sqrt{M}} \pr{s^{-c} \norm{u}_{L^\iny\pr{Q_{s}}}}^\te.
\label{3bgrLHS}
\end{align}

To complete the proof, we need to bound the lefthandside from below using the assumption that $\norm{u}_{L^\iny\pr{Q_1}} \ge \norm{u}_{L^\iny\pr{B_b}} \ge 1$.
We repeat the argument from \cite{KSW15} here.
This assumption implies that there exists $z_0 \in Q_1$ such that $\abs{u\pr{z_0}} \ge 1$. 
Without loss of generality, we'll assume that $u\pr{z_0} \ge 1$.
Since $u$ is real-valued, then for any $a > 0$, we have that either $u\pr{z} \ge a$ for all $z \in Q_{6/5}$, or there exists $z_1 \in Q_{6/5}$ such that $u\pr{z_1} < a$.
We'll need to choose $a$ appropriately.
If the second case holds, then by \eqref{phiBd} we see that
\begin{align*}
\frac{u\pr{z_1}}{\phi\pr{z_1}} \le \frac{a}{\phi\pr{z_1}} \le a \exp\pr{C_1 \sqrt{M}},
\end{align*}
while
\begin{align*}
\frac{u\pr{z_0}}{\phi\pr{z_0}} \ge \exp\pr{-C_1 \sqrt{M}}.
\end{align*}
If we set $a = \frac{1}{2} \exp\pr{- 2 C_1 \sqrt{M}}$ then
\begin{align*}
\frac{u\pr{z_1}}{\phi\pr{z_1}} \le \frac{1}{2}\exp\pr{-C_1 \sqrt{M}}
\end{align*}
and it follows that
\begin{align*}
C \norm{\gr v}_{L^\iny\pr{Q_{6/5}}} 
\ge \abs{v\pr{z_0} - v\pr{z_1}} 
\ge \frac{u\pr{z_0}}{\phi\pr{z_0}} - \frac{u\pr{z_1}}{\phi\pr{z_1}}
\ge \frac{1}{2}\exp\pr{-C_1 \sqrt{M}}.
\end{align*}
Combining this bound with \eqref{3bgrLHS} and Lemma \ref{ZsBounds} leads to the proof of the theorem.
If we are in the former case, then $u\pr{z} \ge a$ for all $z \in Q_{6/5}$ and the conclusion of the theorem is obviously satisfied.
The proof of the Theorem \ref{localEstngW} is now complete.

\begin{rem}
Using the similar ideas as in \cite{KSW15}, one could also study the quantitative Landis conjecture for \eqref{EPDE}, \eqref{EPDEgrW}, \eqref{EPDEngW} defined in an exterior domain.  We leave this generalization to the reader.
\end{rem}

\renewcommand{\theequation}{A.\arabic{equation}}
%\renewcommand{\thelemma}{A.\arabic{lemma}}
%\setcounter{lemma}{0}
%\renewcommand{\thelemma}{\Alph{section}\arabic{lemma}}

% redefine the command that creates the equation no.
\setcounter{equation}{0}  % reset counter 
\section*{Appendix}  % use *-form to suppress numbering

In the appendix, we will prove Lemma~\ref{decompLem}. 
Equivalently, we need to show that
\begin{align*}
\pr{D + \widetilde W} \widetilde D %&= L \\
&= \di A \gr \\
%&= \di \pr{a_{11} \del_x + a_{12} \del_y, a_{12} \del_x + a_{22} \del_y} \\
&= \pr{\del_x a_{11} + \del_{y} a_{12} } \del_x + \pr{\del_x a_{12} + \del_y a_{22}} \del_y + a_{11} \del_{xx} + 2 a_{12} \del_{xy} + a_{22} \del_{yy}.
\end{align*}
To start, we use the notation $\disp \widetilde W = \frac{w_1 + i w_2}{\det\pr{A+I}}.$
Because we assume that $\det A = 1$, then $\nu\pr{z} = 0$ and $D$ is given by \eqref{DDet1Expr}.
Since
\begin{align*}
D \widetilde D 
&= \frac{1}{\det\pr{A+I}}\set{\brac{{1+a_{11}} + i a_{12}}\del_x + \brac{a_{12} + i\pr{1 + a_{22}}}\del_y } \set{\brac{1 + a_{11} - i a_{12}} \del_x + \brac{a_{12} - i\pr{1+a_{22}}} \del_y} \\
%&= \frac{1}{2 + a_{11} + a_{22}} \set{ \brac{{1+a_{11}} + i a_{12}}\del_x \brac{1 + a_{11} - i a_{12}} + \brac{a_{12} + i\pr{1 + a_{22}}}\del_y \brac{1 + a_{11} - i a_{12}}} \del_x \\
%&+ \frac{1}{2 + a_{11} + a_{22}} \set{ \brac{{1+a_{11}} + i a_{12}} \del_x \brac{a_{12} - i\pr{1+a_{22}}} + \brac{a_{12} + i\pr{1 + a_{22}}} \del_y \brac{a_{12} - i\pr{1+a_{22}} }} \del_y \\
%&+ \frac{\brac{{1+a_{11}} + i a_{12}}\brac{1 + a_{11} - i a_{12}}}{2 + a_{11} + a_{22}} \del_{xx} 
%+ \frac{ \brac{a_{12} + i\pr{1 + a_{22}}}\brac{a_{12} - i\pr{1+a_{22}}}}{2 + a_{11} + a_{22}}\del_{yy} \\
%&+ \frac{\brac{{1+a_{11}} + i a_{12}} \brac{a_{12} - i\pr{1+a_{22}}} +  \brac{a_{12} + i\pr{1 + a_{22}}}\brac{1 + a_{11} - i a_{12}}}{2 + a_{11} + a_{22}} \del_{xy} \\
&= \frac{1}{\det\pr{A+I}} \brac{ \pr{1+a_{11}} \del_x a_{11} + a_{12} \del_x a_{12}  
+ a_{12}\del_y a_{11}  + \pr{1 + a_{22}} \del_y a_{12} } \del_x \\
&+ \frac{i}{\det\pr{A+I}} \brac{ - \pr{1+a_{11}} \del_x a_{12} + a_{12} \del_x a_{11} - a_{12} \del_y a_{12} + \pr{1 + a_{22}} \del_y a_{11} } \del_x \\
&+ \frac{1}{\det\pr{A+I}} \brac{ \pr{1+a_{11}}\del_x a_{12} + a_{12} \del_x a_{22}  + a_{12} \del_y a_{12} + \pr{1 + a_{22}}  \del_y a_{22} }   \del_y \\
&+ \frac{i}{\det\pr{A+I}} \brac{ - \pr{1+a_{11}}\del_x a_{22} + a_{12} \del_x a_{12} - a_{12} \del_y a_{22} + \pr{1 + a_{22}}\del_y a_{12} } \del_y \\
&+ a_{11} \del_{xx} + 2 a_{12} \del_{xy}+ a_{22}\del_{yy} ,
\end{align*}
and
\begin{align*}
\widetilde W  \widetilde D 
&= \frac{w_1 + i w_2}{\det\pr{A+I}}\set{\brac{1 + a_{11} - i a_{12}} \del_x + \brac{a_{12} - i\pr{1+a_{22}}} \del_y} \\
&= \frac{1}{\det\pr{A+I}} \set{\brac{w_1\pr{1 + a_{11}} + w_2 a_{12}} + i \brac{w_2\pr{1 + a_{11}} - w_1 a_{12}}} \del_x \\
&+ \frac{1}{\det\pr{A+I}} \set{\brac{ w_1 a_{12} + w_2\pr{1 + a_{22}}} + i \brac{w_2 a_{12} - w_1\pr{1 + a_{22}} }} \del_y,
\end{align*}
then it suffices to show that the following four equations are satisfied
\begin{align*}
\det\pr{A+I} \pr{ \del_x a_{11} + \del_{y} a_{12} }
&=  \pr{1+a_{11}} \del_x a_{11} + a_{12} \del_x a_{12}  + a_{12}\del_y a_{11}  + \pr{1 + a_{22}} \del_y a_{12}
+ \pr{1+a_{11}} w_1 + a_{12} w_2 \\
0 &=  - \pr{1+a_{11}} \del_x a_{12} + a_{12} \del_x a_{11} - a_{12} \del_y a_{12} + \pr{1 + a_{22}} \del_y a_{11}
+ \pr{1+a_{11}} w_2 - a_{12} w_1 \\
\det\pr{A+I} \pr{\del_x a_{12} + \del_y a_{22} }
&= \pr{1+a_{11}}\del_x a_{12} + a_{12} \del_x a_{22}  + a_{12} \del_y a_{12} + \pr{1 + a_{22}}  \del_y a_{22}    
+ a_{12} w_1 + \pr{1 + a_{22}} w_2 \\
0 &= - \pr{1+a_{11}}\del_x a_{22} + a_{12} \del_x a_{12} - a_{12} \del_y a_{22} + \pr{1 + a_{22}}\del_y a_{12} 
+ a_{12} w_2 - \pr{1 + a_{22}} w_1 .
\end{align*}

Since $\det A = 1$, then $a_{11} a_{22} - a_{12}^2 = 1$ and
\begin{align}
\del_x a_{22} &= \frac{2 a_{12} \del_x a_{12} - a_{22} \del_x a_{11}  }{a_{11}} 
\label{a22xDer} \\
\del_y a_{22} &= \frac{2 a_{12} \del_y a_{12} - a_{22} \del_y a_{11}  }{a_{11}}.
\label{a22yDer}
\end{align}
Replacing the derivatives of $a_{22}$ with \eqref{a22xDer} and \eqref{a22yDer}, then simplyfing, the four equations above are equivalent to
\begin{align*}
\pr{1+a_{11}} w_1 + a_{12} w_2
&=\pr{ 1 + a_{22}} \del_x a_{11} - a_{12} \del_x a_{12}  - a_{12}\del_y a_{11}  + \pr{1 + a_{11} } \del_{y} a_{12}  \\
\pr{1+a_{11}} w_2 - a_{12} w_1 
&=- a_{12} \del_x a_{11} + \pr{1+a_{11}} \del_x a_{12} - \pr{1 + a_{22}} \del_y a_{11} + a_{12} \del_y a_{12} \\
a_{12} w_1 + \pr{1 + a_{22}} w_2
&= \frac{a_{12} a_{22} }{a_{11}}\del_x a_{11} + \pr{\frac{1 + a_{11} - a_{12}^2}{a_{11}}} \del_x a_{12} - \frac{\pr{1 + a_{11}}a_{22}}{a_{11}} \del_y a_{11} +  \frac{\pr{2 + a_{11}} a_{12}}{a_{11}} \del_y a_{12}  \\
\pr{1 + a_{22}} w_1 -a_{12} w_2
&= \frac{\pr{1+a_{11}}a_{22} }{a_{11}}\del_x a_{11} - \frac{ \pr{2+a_{11}} a_{12} }{a_{11}} \del_x a_{12} + \frac{a_{12}a_{22}}{a_{11}} \del_y a_{11} + \frac{1 + a_{11} -a_{12}^2}{a_{11}} \del_y a_{12}.
\end{align*}

Noting that $\pr{1+a_{11}}^2 + a_{12}^2 = a_{11} \det\pr{A+I}$, the first pair of equations may be solved for $w_1$ and $w_2$:
\begin{align*}
\brac{\begin{array}{c} w_1 \\ w_2 \end{array}} 
&= \frac{1}{a_{11} \det\pr{A+I}} \brac{\begin{array}{cc} 1+ a_{11} & -a_{12} \\ a_{12} & 1 + a_{11} \end{array}} \brac{\begin{array}{cc} 
\pr{ 1 + a_{22}} \del_x a_{11} - a_{12} \del_x a_{12}  - a_{12}\del_y a_{11}  + \pr{1 + a_{11} } \del_{y} a_{12} \\
- a_{12} \del_x a_{11} + \pr{1+a_{11}} \del_x a_{12} - \pr{1 + a_{22}} \del_y a_{11} + a_{12} \del_y a_{12} 
\end{array}} \\
&=  \brac{\begin{array}{cc} 
\frac{ a_{11} + a_{22} + 2a_{11}a_{22}}{a_{11} \det\pr{A+I}} \del_x a_{11} - \frac{2\pr{1+a_{11}}a_{12}}{a_{11} \det\pr{A+I}} \del_x a_{12}  + \frac{a_{12}\pr{a_{22} - a_{11}}}{a_{11} \det\pr{A+I}} \del_y a_{11}  + \frac{\pr{1 + a_{11} }^2 - a_{12}^2}{a_{11} \det\pr{A+I}} \del_{y} a_{12} \\
\frac{a_{12}\pr{a_{22} - a_{11}}}{a_{11} \det\pr{A+I}} \del_x a_{11} + \frac{\pr{1+a_{11}}^2 - a_{12}^2}{a_{11} \det\pr{A+I}} \del_x a_{12} - \frac{a_{11} + a_{22} + 2a_{11}a_{22}}{a_{11} \det\pr{A+I}} \del_y a_{11} + \frac{2\pr{1 + a_{11}}a_{12}}{a_{11} \det\pr{A+I}} \del_y a_{12} 
\end{array}},
\end{align*}
which is consistent with the definition of $\widetilde W$ given in the statement of the lemma.

Similarly, since $\pr{1+a_{22}}^2 + a_{12}^2 = a_{22} \det\pr{A+I}$, the second pair of equations also implies that
\begin{align*}
\brac{\begin{array}{c} w_1 \\ w_2 \end{array}} 
&= \frac{1}{a_{22} \det\pr{A+I}} \brac{\begin{array}{cc} 1+ a_{22} & a_{12} \\ -a_{12} & 1 + a_{22} \end{array}} \\
& \times \brac{\begin{array}{cc} 
\frac{\pr{1+a_{11}}a_{22} }{a_{11}}\del_x a_{11} - \frac{ \pr{2+a_{11}} a_{12} }{a_{11}} \del_x a_{12} + \frac{a_{12}a_{22}}{a_{11}} \del_y a_{11} + \frac{1 + a_{11} -a_{12}^2}{a_{11}} \del_y a_{12} \\
\frac{a_{12} a_{22} }{a_{11}}\del_x a_{11} + \pr{\frac{1 + a_{11} - a_{12}^2}{a_{11}}} \del_x a_{12} - \frac{\pr{1 + a_{11}}a_{22}}{a_{11}} \del_y a_{11} +  \frac{\pr{2 + a_{11}} a_{12}}{a_{11}} \del_y a_{12}
\end{array}} \\
&=  \brac{\begin{array}{cc} 
\frac{ a_{11} + a_{22} + 2a_{11}a_{22}}{a_{11} \det\pr{A+I}} \del_x a_{11} - \frac{2\pr{1+a_{11}}a_{12}}{a_{11} \det\pr{A+I}} \del_x a_{12}  + \frac{a_{12}\pr{a_{22} - a_{11}}}{a_{11} \det\pr{A+I}} \del_y a_{11}  + \frac{\pr{1 + a_{11} }^2 - a_{12}^2}{a_{11} \det\pr{A+I}} \del_{y} a_{12} \\
\frac{a_{12}\pr{a_{22} - a_{11}}}{a_{11} \det\pr{A+I}} \del_x a_{11} + \frac{\pr{1+a_{11}}^2 - a_{12}^2}{a_{11} \det\pr{A+I}} \del_x a_{12} - \frac{a_{11} + a_{22} + 2a_{11}a_{22}}{a_{11} \det\pr{A+I}} \del_y a_{11} + \frac{2\pr{1 + a_{11}}a_{12}}{a_{11} \det\pr{A+I}} \del_y a_{12} 
\end{array}}.
\end{align*}
This completes the proof of the decomposition lemma.

\bibliography{refs}
\bibliographystyle{plain}

\end{document}